\documentclass[12pt]{amsart}

  \usepackage{graphicx}
  \usepackage{latexsym,amscd,amssymb,epsfig,a4wide}

  \def\sm{\smallsetminus}

  \def\aA{{\mathcal A}}
  \def\bB{{\mathcal B}}
  \def\cC{{\mathcal C}}
  \def\eE{{\mathcal E}}
  \def\fF{{\mathcal F}}
  \def\gG{{\mathcal G}}
  \def\lL{{\mathcal L}}
  \def\sS{{\mathcal S}}
  \def\sss{{\mathfrak S}}

  \def\RR{{\mathbb R}}
  \def\ZZ{{\mathbb Z}}
  \def\oo{{\mathfrak o}}
  \def\xxx{{\mathfrak X}}
  \def\yyy{{\mathfrak Y}}
  \def\AO{{\mathrm{AO}}}
  \def\codim{{\mathrm{codim}}}
  \def\Des{{\mathrm{Des}}}
  \def\cDes{{\mathrm{cDes}}}
  \def\sing{{\mathrm{sing}}}
  \def\supp{{\mathrm{supp}}}

  \def\tr{{\mathrm{tr}}}
  \def\TV{{\mathrm{TV}}}

  \theoremstyle{plain}
    \newtheorem{theorem}{Theorem}[section]
    \newtheorem{proposition}[theorem]{Proposition}
    \newtheorem{lemma}[theorem]{Lemma}
    \newtheorem{corollary}[theorem]{Corollary}
    
  \theoremstyle{definition}
    
    \newtheorem{example}[theorem]{Example}

    \newtheorem{remark}[theorem]{Remark}

  \numberwithin{equation}{section}

\begin{document}

\title[Functions of hyperplane walks]{Functions of random walks on
hyperplane arrangements}

  \author{Christos~A.~Athanasiadis}
  \address{Department of Mathematics \\
           University of Athens \\
           Panepistimioupolis, Athens 15784 \\
           Greece}
\email{caath@math.uoa.gr}

  \author{Persi~Diaconis}
  \address{Department of Mathematics and Statistics\\
           Stanford University\\
           Stanford, CA 94305}

\date{December 14, 2009; Revised, February 5, 2010}

  \begin{abstract}
    Many seemingly disparate Markov chains are unified when viewed as random
    walks on the set of chambers of a hyperplane arrangement. These include
    the Tsetlin library of theoretical computer science and various shuffling
    schemes. If only selected features of the chains are of interest, then the
    mixing times may change. We study the behavior of hyperplane walks, viewed
    on a subarrangement of a hyperplane arrangement. These include many new
    examples, for instance a random walk on the set of acyclic orientations
    of a graph. All such walks can be treated in a uniform fashion, yielding
    diagonalizable matrices with known eigenvalues, stationary distribution and
    good rates of convergence to stationarity.
  \end{abstract}

  \maketitle

  \section{Introduction}
  \label{sec:intro}

  Many seemingly disparate Markov chains may be successfully studied by
  viewing them as random walks on the set of chambers of a hyperplane
  arrangement \cite{BHR}. These include the Tsetlin library of theoretical
  computer science, a variety of walks on the hypercube and various shuffling
  schemes \cite{BD}. If only selected features of such a Markov chain are of
  interest (for instance, only a few sites on the hypercube or the relative
  ordering of the top few cards), then the mixing time may change. Following
  a suggestion of Uyemura Reyes \cite{Rey}, we study the behavior of hyperplane
  walks, viewed on subarrangements of a given hyperplane arrangement. This
  leads to new Markov chains which permit a full analysis. The following two
  examples illustrate our results and are used as running examples throughout.

  \begin{example}  [Conquering Territory]  \label{ex:grid}
    {\rm Consider an $m \times n$ grid, with each node labeled with $+1$
    or $-1$. At each stage, a node is chosen from a fixed probability
    distribution, then a neighborhood of this node is chosen and finally,
    all labels of the nodes in this neighborhood are changed to $+1$ or
    all are changed to $-1$, according to a specific distribution. As
    explained in Section \ref{sec:appA}, this Markov chain can be viewed as
    a hyperplane walk on the Boolean arrangement. Such walks were first
    studied in \cite{BHR, BD} and include the classical Ehrenfest urn. The
    stationary distribution depends on the various probabilistic specifications
    but the theory of \cite{BHR, BD}, reviewed in Section \ref{sec:back}, gives
    a useful description of this distribution, as well as of the eigenvalues and
    rates of convergence to stationarity.

    Suppose now that only the labels of a few nodes (for instance, the
    four corners or the middle node) are of interest. Common sense suggests
    that the induced process on these nodes may converge to stationarity
    at a faster rate than the entire chain. For example, in the Ehrenfest urn
    with $n$ particles, order of $n \log n$ steps are required to equilibriate
    on the full state space but order of $n$ steps suffice for a few tagged
    particles. Further details and examples appear in the sequel. \qed}
  \end{example}

  Sometimes the induced chain is the object of direct interest, with the
  original chain opaque in the background. This is the case in our second
  example.

  \begin{example} [Acyclic Orientations]  \label{ex:acyclic}
    {\rm Let $\gG$ be a simple undirected graph. A Markov chain on the set
    of acyclic orientations of $\gG$ can be defined as follows: At each stage,
    a node $v$ of $\gG$ is chosen from a fixed probability distribution $w$
    and all edges of $\gG$ incident to $v$ are oriented inward, towards $v$.
    Under some mild assumptions on $\gG$ and $w$, this is an ergodic Markov
    chain on the set of acyclic orientations of $\gG$ with describable
    stationary distribution and eigenvalues and with good control on rates
    of convergence. It arises as the chain induced from the Tsetlin library
    on the braid arrangement, where the subarrangement is the graphical
    arrangement corresponding to $\gG$. It also arises as a walk on the
    Boolean arrangement; see Section \ref{sec:appA} for a detailed
    discussion. \qed}
  \end{example}

  This paper is organized as follows. Section \ref{sec:back} includes background
  on hyperplane walks and functions of a Markov chain, along with an overview of
  the basic examples of hyperplane walks on the Boolean and braid arrangements. Our
  main results appear in Section \ref{sec:main}. Developing a suggestion in
  \cite{Rey}, the process induced from a hyperplane walk on the set of chambers
  of a subarrangement is considered. Although a function of a Markov chain is
  usually not Markov, it is shown that subarrangement processes are Markov
  chains. Moreover, the subarrangement chains are shown to be hyperplane walks
  in their own right. This implies that the whole tool kit of results for
  hyperplane walks \cite{BHR, BD} is available. One striking feature of general
  hyperplane walks is that they have nonnegative real eigenvalues, although
  these chains are almost never symmetric or reversible. A purely combinatorial
  proof of this fact, as well as a new coupling proof of the basic theorem of
  \cite{BD}, giving rates of convergence to stationarity, also appear in Section
  \ref{sec:main}.

  Sections \ref{sec:appA} and \ref{sec:appB} give applications of the general theory
  to the hyperplane walks of Section \ref{sec:back}, treating Examples \ref{ex:grid}
  and \ref{ex:acyclic} in several variations. These include a variety of functions on
  the Tsetlin library and inverse $a$-shuffling Markov chains on the symmetric group,
  such as those assigning to a permutation the set of elements preceding a given entry
  in its linear representation, the descent set and the cyclic descent set. Most of the
  induced chains we study seem very different from their parents. As a byproduct of
  our considerations, we mention an interpretation for the unsigned coefficients of
  the chromatic polynomial of a graph as the multiplicities of the transition matrix
  of a certain natural Markov chain on the set of acyclic orientations of $\gG$
  (see Proposition \ref{prop:a}). Section \ref{sec:semi} briefly discusses
  extensions to random walks on semigroups.

  \section{Background}
  \label{sec:back}

  This section reviews the main results on hyperplane walks, develops needed
  examples (hypercube walks, Tsetlin library, inverse $a$-shuffles) and provides
  background on functions of a Markov chain.

  \subsection{Hyperplane walks}
  \label{sec:hyper}

  Let $\aA$ be a hyperplane arrangement in $V = \RR^n$, meaning a finite set
  of affine hyperplanes in $V$. The \emph{intersection poset} of $\aA$ is the
  set $\lL_\aA = \{ \cap \, \eE: \eE \subseteq \aA\}$, consisting of all
  affine subspaces of $V$ which can be written as intersections of some of
  the hyperplanes of $\aA$, partially ordered by reverse inclusion. The space
  $V$, corresponding to $\eE = \varnothing$, is the minimum element of $\lL_\aA$.

  The connected components of the space obtained from $V$ by removing the union
  of the hyperplanes of $\aA$ are called \emph{chambers}. The \emph{restriction}
  of $\aA$ on an intersection subspace $W \in \lL_\aA$ is the hyperplane
  arrangement in $W$ consisting of the intersections of $W$ with the
  hyperplanes of $\aA$ which are not parallel to (in particular, do not
  contain) $W$. The chambers of all such restricted arrangements are called
  \emph{faces} of $\aA$. Thus the chambers of $\aA$ are exactly its faces of
  dimension $n$. We will denote by $\cC_\aA$ the set of chambers of $\aA$ and
  by $\fF_\aA$ its set of faces. The elements of $\fF_\aA$ are the open cells
  of a regular cell decomposition of $V$; see \cite[Section 2.1]{OMbook}.
  Figure \ref{fig:FC} shows an arrangement of four hyperplanes (lines) in
  $\RR^2$ which has ten chambers, thirteen one-dimensional faces (open line
  segments) and four zero-dimensional faces (points).

  \begin{figure}[h]
  \begin{center}
  \includegraphics[width=2.5in]{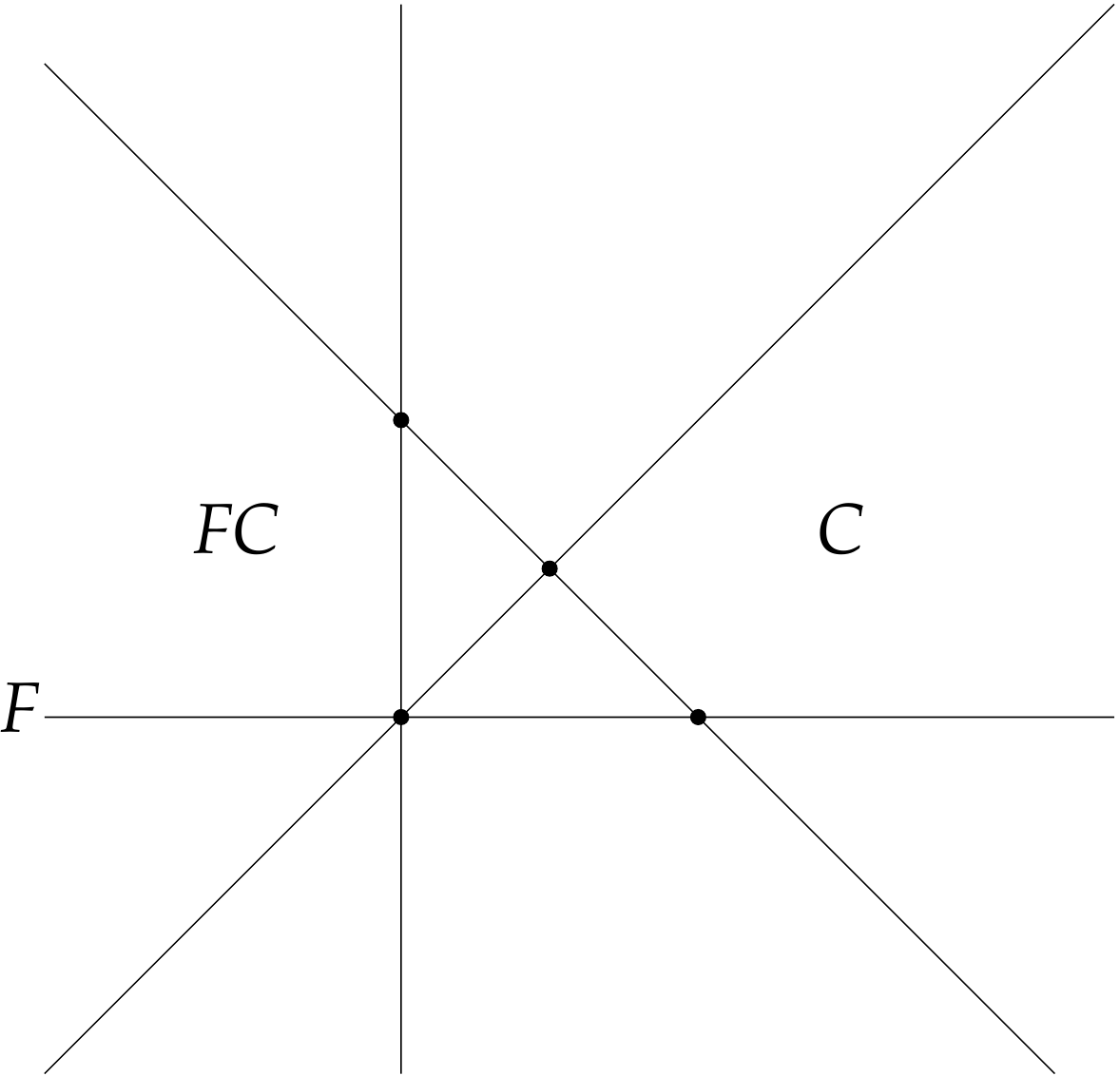}
  \end{center}
  \label{fig:FC}
  \end{figure}

  Given faces $F, G \in \fF_\aA$, we say that $F$ is a face of $G$ if $F$ is
  contained in the closure of $G$, with respect to the standard topology on
  $V$. Given a chamber $C \in \cC_\aA$ and a face $F \in \fF_\aA$, a lemma of
  Tits \cite{Ti} asserts that among all chambers of $\aA$ which have $F$ as
  a face, there is a unique chamber $C'$ which is closest to $C$, in the
  sense that the set of hyperplanes in $\aA$ separating $C'$ from $C$ is
  minimum with respect to inclusion. We will refer to this chamber $C'$ as
  the \emph{projection} of $C$ on $F$ and will denote it by $FC$. For an
  alternative definition, pick any points $x \in F$ and $y \in C$ and move by
  a small distance away from $x$ in the direction of $y$, in the line segment
  joining these two points. The resulting point lies in a well defined face
  of $\aA$, which is the face $FC$. An example is given in Figure \ref{fig:FC}.
  The second recipe can be used to define the face $F G \in \fF_\aA$ for
  any two faces $F, G \in \fF_\aA$.

  Using these ingredients, Bidigare, Hanlon and Rockmore \cite{BHR} suggested
  the following family of Markov chains on the state space $\cC_\aA$. Start
  with a probability measure $w$ on $\fF_\aA$. A step in the chain is given by
  choosing a face $F \in \fF_\aA$ from $w$ and moving from the current chamber
  $C \in \cC_\aA$ to $FC$.  Equivalently, we can describe this chain by
  defining its transition matrix $K$ by the equation
     \begin{equation} \label{eq:BHRdef}
     K(C, C') \ = \ \sum_{F \in \fF_\aA: \, FC=C'} \ w(F).
    \end{equation}
  Henceforth, we will refer to this Markov chain as the \emph{hyperplane walk}
  on $\aA$ (or on $\cC_\aA$) associated to $w$.

  Hyperplane walks being so general, it is surprising that there is a
  relatively complete theory for them. We recall that the \textit{total variation
  distance} between two probability distributions $P$ and $Q$ on a finite set
  $\Omega$ is defined as
    $$ \| P - Q \|_\TV \ = \ \max_{A \subseteq \Omega} \, |P(A) - Q(A)|. $$
  The measure $w$ on $\fF_\aA$ is said to be \emph{separating} \cite{BD} if
  for every $H \in \aA$ there exists a face $F \in \fF_\aA$ such that $F
  \not\subseteq H$ and $w(F) > 0$. We combine the main results of \cite{BHR, BD}
  into the following theorem.

  \begin{theorem} \label{thm:BHR-BD}
    Let $\aA$ be a hyperplane arrangement in $V$ with set of faces $\fF_\aA$
    and intersection poset $\lL_\aA$ and let $w$ be a probability measure on
    $\fF_\aA$.
     \begin{itemize}
        \item[(i)] The characteristic polynomial of $K$ is given by
        $$\det (xI - K) \ = \ \prod_{W \in \lL_\aA} \ (x - \lambda_W)^{m_W},$$
        where
          \begin{equation} \label{eq:lambdaW}
            \lambda_W \ = \ \sum_{F \in \fF_\aA: \, F \subseteq W} \ w(F)
          \end{equation}
        is an eigenvalue,
        $$m_W \ = \ |\mu_\aA (V, W)| \ = \ (-1)^{\codim(W, V)} \mu_\aA(V, W), $$
        $\mu_\aA$ is the M\"obius function of $\lL_\aA$ and $\codim(W, V)$ is
        the codimension of $W$ in $V$.

        \item[(ii)] The matrix $K$ is diagonalizable.
        \item[(iii)] $K$ has a unique stationary distribution $\pi$ if and only
        if $w$ is separating.
        \item[(iv)] Assume that $w$ is separating and let $K_C^l$ be the
        distribution of the chain started from the chamber $C$ after $l$ steps.
        Then its total variation distance from $\pi$ satisfies
          \begin{equation} \label{eq:tot1}
          \| K_C^l - \pi \|_\TV \ \le \ P \{F_1 F_2 \cdots F_l \not\in \cC_\aA\},
          \end{equation}
        where $(F_1, F_2,\dots)$ consists of independent and identically
        distributed picks from the measure $w$ on $\fF_\aA$, and
          \begin{equation} \label{eq:tot2}
          \| K_C^l - \pi \|_\TV \ \le \ \sum_{H \in \aA} \ \lambda_H^l.
          \end{equation}
     \end{itemize}
  \end{theorem}

  Furthermore, there is a useful description of the stationary distribution
  which will not be detailed here; see \cite[Theorem 2 (b)]{BD}. As noticed
  in \cite[Section 4]{BD}, the right-hand side of (\ref{eq:tot1}) is bounded
  from above by that of (\ref{eq:tot2}). The bounds in (\ref{eq:tot1}) and
  (\ref{eq:tot2}) are usually good but not perfect.

  \subsection{Examples}
  \label{sec:ex}

  Our main examples involve the Boolean and braid arrangements. In what
  follows, we denote by $\sss_n$ the symmetric group of permutations of the
  set $[n] := \{1, 2,\dots,n\}$. We will often use the one line notation
  $(\tau(1), \tau(2),\dots,\tau(n))$ for a permutation $\tau \in \sss_n$. It
  will be helpful to think of $\tau$ as a linear ordering of a deck of $n$
  cards, labeled bijectively by the elements of $[n]$.

  \medskip
  \noindent
  \textbf{A. The Boolean arrangement}. Let $\aA = \{H_1, H_2,\dots,H_n\}$ be
  the \emph{Boolean arrangement} in $V = \RR^n$, where $H_i$ is the coordinate
  hyperplane in $\RR^n$ defined by the equation $x_i = 0$, for $1 \le i \le
  n$. The intersection poset $\lL_\aA$ is isomorphic to the lattice of
  subsets of the set $[n]$, ordered by inclusion, where the isomorphism maps
  a subspace $W \in \lL_\aA$ to the set of indices $i \in [n]$ for which $W
  \subseteq H_i$. The M\"obius function $\mu_\aA$ of this poset satisfies
  $\mu_\aA (V, W) = (-1)^{\codim(W, V)}$ for $W \in \lL_\aA$.

  The set $\fF_\aA$ of faces of $\aA$ is in bijection with the set $\{-, 0,
  +\}^n$ of all $3^n$ possible sign vectors of length $n$ (where the bijection
  is induced by the map assigning to each point $x \in \RR^n$ the sequence of
  signs of the coordinates of $x$). The chambers of $\aA$ are the orthants in
  $\RR^n$; they correspond to the $2^n$ elements of $\{-, +\}^n$. Using these
  bijections, we may identify $\cC_\aA$ and $\fF_\aA$ with the sets $\{-, +\}^n$
  and $\{-, 0, +\}^n$, respectively (the former may also be identified with
  the set of vertices of the $n$-dimensional cube $[-1, 1]^n$).
  The projection $FC$ of a chamber $C \in \cC_\aA$ on a face $F \in \fF_\aA$
  is the chamber whose sign vector is obtained from that of $F$ by switching
  every zero coordinate to the corresponding coordinate of $C$. Thus, given a
  probability measure $w$ on $\fF_\aA$, the hyperplane walk on $\aA$
  associated to $w$ proceeds from the current chamber $C \in \cC_\aA$ by
  selecting a face $F \in \fF_\aA$ with probability $w(F)$ and replacing the
  $i$th coordinate of $C$ by the $i$th coordinate of $F$, whenever the latter
  is nonzero, to reach the next chamber in the walk. Some examples of these
  walks appear in \cite[Section 2.3]{BHR} and \cite[Section 3B]{BD}.

  It follows from Theorem \ref{thm:BHR-BD} that the transition matrix $K$ for
  this chain has eigenvalues
    \begin{equation} \label{eq:lambdaB}
      \lambda_S \ = \ \sum_{F \in \fF_S} \ w(F),
    \end{equation}
  one for each $S \subseteq [n]$, where $\fF_S$ denotes the set of faces of $\aA$
  whose sign vectors have their $i$th coordinate equal to zero for every $i
  \in S$, and that $K$ has a unique stationary distribution $\pi$ if and only if
  for every $1 \le i \le n$ there exists a face $F \in \fF_\aA$ with nonzero
  $i$th coordinate, such that $w(F) > 0$.

  The eigenvalues which correspond to the hyperplanes $H_i$ and which appear
  in the right-hand side of (\ref{eq:tot2}), are the numbers $\lambda_{\{i\}}$
  for $1 \le i \le n$.

  \begin{example} \label{ex:nearest}
  {\rm For $\varepsilon \in \{-, +\}$ and $1 \le i \le n$, we denote by
  $E_i^\varepsilon$ the face of $\aA$ whose sign vector has its $i$th
  coordinate equal to $\varepsilon$ and all other coordinates equal to zero.
  We choose face weights so that $w(E_i^\varepsilon) = w_i^\varepsilon$ for
  $\varepsilon \in \{-, +\}$ and $1 \le i \le n$, and $w(F) = 0$ for all other
  faces $F \in \fF_\aA$, where the $w_i^\varepsilon$ are nonnegative real
  numbers summing to 1.
  The resulting Markov chain is the nearest neighbor random walk on the vertex
  set $\{-, +\}^n$ of the $n$-dimensional cube, which evolves by picking a
  coordinate $i$, at each stage, and switching it to $\varepsilon$, with
  probability $w_i^\varepsilon$. Thus, the transition matrix $K$ for this chain
  is given by
    $$K (x, x') \ = \ \begin{cases} \sum_{i=1}^n \, w_i^{\varepsilon_i}, & \text{if \
                      $x'=x$, }\\
                      w_i^{-\varepsilon_i}, & \text{if \ $x'$ is
                      obtained from $x$ by switching the $i$th coordinate to
                      $-\varepsilon_i$, }\\
                     0, & \text{otherwise} \end{cases} $$
  for $x, x' \in \{-, +\}^n$ with $x = (\varepsilon_1,
  \varepsilon_2,\dots,\varepsilon_n)$. It has eigenvalues
    $$ \lambda_S \ = \ \sum_{i \in [n] \sm S} \ w_i, $$
  one for each $S \subseteq [n]$, where $w_i = w^-_i + w^+_i$. Moreover, $K$
  has a unique stationary distribution $\pi$ if and only if $w_i>0$ holds for
  every $1 \le i \le n$. In that case, $\pi$ is given by the formula
     \begin{equation} \label{eq:pi-nearest}
       \pi(x) \ = \ \prod_{i=1}^n \ \frac{w_i^{\varepsilon_i}}{w_i}
     \end{equation}
  for $x = (\varepsilon_1, \varepsilon_2,\dots,\varepsilon_n) \in \{-, +\}^n$
  (see, for instance, \cite[Section 3B]{BD}) and (\ref{eq:tot2}) gives the
  bound
    \begin{equation} \label{eq:tot-nearest}
      \| K_x^l - \pi \|_\TV \ \le \ \sum_{i=1}^n \ (1 - w_i)^l,
    \end{equation}
  where $K_x^l$ is the distribution of the chain started from $x$ after $l$
  steps. \qed}
  \end{example}

  \medskip
  \noindent
  \textbf{B. The braid arrangement}. Let $\aA$ be the \emph{braid arrangement}
  in $\RR^n$, consisting of the ${n \choose 2}$ hyperplanes defined by
  the equations $x_i - x_j = 0$ for $1 \le i < j \le n$. The intersection
  poset $\lL_\aA$ is isomorphic to the lattice of partitions of the set $[n]$,
  ordered by refinement. The isomorphism maps a subspace $W \in \lL_\aA$ to
  the partition of $[n]$ in which $i$ and $j$ are in the same block if and
  only if $x_i = x_j$ holds for every point $(x_1, x_2,\dots,x_n) \in W$.

  The set $\cC_\aA$ of chambers of $\aA$ is in bijection with $\sss_n$ and the
  set $\fF_\aA$ of faces is in bijection with the set
  of ordered partitions of $[n]$, meaning set partitions $(B_1, B_2,\dots,B_k)$
  of $[n]$ in which the order of the blocks matters. To be specific, let
  us agree that the permutation $\tau \in \sss_n$ corresponds to the chamber
    $$ x_{\tau(1)} > x_{\tau(2)} > \cdots > x_{\tau(n)}. $$
  More generally, the ordered partition $B = (B_1, B_2,\dots,B_k)$ of
  $[n]$ corresponds to the face of $\aA$ defined by the equalities $x_i
  = x_j$, whenever $i$ and $j$ lie in the same block of $B$, and the
  inequalities $x_i > x_j$, whenever the block of $B$ which contains
  $i$ has smaller index than the block of $B$ which contains $j$.

  The action of faces on chambers can be described as follows. Let $C
  \in \cC_\aA$ be the chamber corresponding to $\tau \in \sss_n$ and $F
  \in \fF_\aA$ be the face corresponding to the ordered partition $B$ of
  $[n]$. One can then check that $FC$ is the chamber which corresponds
  to the permutation of $[n]$ which is obtained from $B$ by listing the
  elements of each block of $B$ in the relative order in which they appear
  in $\tau$. For instance, if $n=9$,
  $\tau = (8, 1, 4, 9, 7, 2, 6, 3, 5)$ and $B = (\{6, 9\}, \{1, 3, 7\},
  \{4\}, \{2, 5, 8\})$, then the resulting permutation is equal to $(9,
  6, 1, 7, 3, 4, 8, 2, 5)$. In the sequel, we identify faces
  (respectively, chambers) of $\aA$ with the corresponding ordered
  partitions (respectively, permutations) of the set $[n]$.

  Different choices of probability measure on $\fF_\aA$ lead to various
  interesting Markov chains on $\sss_n$. We concentrate on the following
  two examples.

  \begin{example} [Tsetlin Library]  \label{ex:tsetlin}
    {\rm Let $w_1, w_2,\dots,w_n$ be nonnegative real numbers summing to 1.
    Choose face weights so that
      $$ w(B) \ = \ \begin{cases} w_i, & \text{if \ $B=(\{i\}, [n] \sm \{i\})$
                    for some $1 \le i \le n$} \\
                    0, & \text{otherwise} \end{cases}$$
    for an ordered partition $B$ of $[n]$. The projection of $\tau$ on $B =
    (\{i\}, [n] \sm \{i\})$ removes the entry $i$ in the one line notation of the
    permutation $\tau$ and places it in front. Hence the transition matrix $K$ is
    the $n! \times n!$ matrix defined by
      $$K(\tau, \tau') \ = \ \begin{cases} w_i, & \text{if \ $\tau'$ is
         obtained from $\tau$ by moving $i$ in front, for some $i$}\\
                    0, & \text{otherwise}. \end{cases}$$
    This chain has been extensively studied as a model of dynamic storage allocation;
    see \cite{DF} for a survey and \cite{BHP} for recent results. It was shown by
    Phatarfod \cite{Ph}, and follows easily from Theorem \ref{thm:BHR-BD} (see \cite{BHR,
    BD}), that $K$ is diagonalizable with eigenvalues
      \begin{equation} \label{eq:phat}
        \lambda_\tau \ = \ \sum_{\tau(i)=i} \ w_i,
      \end{equation}
    one for each $\tau \in \sss_n$. Moreover, $K$ has a unique stationary
    distribution $\pi$ if and only if we have $w_i=0$ for at most one index $i$.
    In that case, $\pi$ is given by ``sampling the weights without replacement
    to generate a random permutation". Thus we have
      \begin{equation} \label{eq:luce}
        \pi(\tau) \ = \ \frac{w_{\tau(1)}w_{\tau(2)} \cdots w_{\tau(n-1)}}
         {(1 - w_{\tau(1)}) (1 - w_{\tau(1)} - w_{\tau(2)}) \cdots
         (1 - w_{\tau(1)} - \cdots - w_{\tau(n-2)}) }
      \end{equation}
    for $\tau \in \sss_n$. This stationary distribution is known as the
    \emph{Luce model} in cognitive psychology; see \cite[p.~174]{Di} for extensive
    references. Equation (\ref{eq:tot2}) of Theorem \ref{thm:BHR-BD} (iv) gives
    the bound
      \begin{equation} \label{eq:tse1}
        \| K_\tau^l - \pi \|_\TV \ \le \ \sum_{1 \le i < j \le n} \ (1 - w_i
         - w_j)^l
      \end{equation}
    on the rate of convergence to stationarity, where $K_\tau^l$ is the distribution
    of the chain started at $\tau$ after $l$ steps. Suppose, for instance, that
    $w_i =1/n$ for $1 \le i \le n$, so that at each stage of the chain, an entry
    of the current permutation is selected uniformly at random and moved in front
    (thus the chain evolves by the ``random to top" rule). Then we have
      \begin{equation} \label{eq:tse2}
        \| K_\tau^l - \pi \|_\TV \ \le \ {n \choose 2} \left(1 - \frac{2}{n}
         \right)^l.
      \end{equation}
    The expression on the right is bounded above by $e^{-2c}/2$ if $l \ge n (\log n
    +c)$, for $c>0$. In this case there is a matching lower bound which shows that
    $n (\log n + c)$ steps are in fact necessary and sufficient for convergence to
    stationarity; see \cite{DFP} for further details and more refined asymptotics.
    \qed}
  \end{example}

  \begin{example} [Inverse $a$-shuffles]  \label{ex:a}
    {\rm Ordinary riffle shuffles have received a careful analysis in \cite{BaD}.
    A key to this analysis is a natural model on inverse riffle shuffles. Informally,
    begin with a deck of cards in order. Label the back of each card by one of the
    numbers in $\{1, 2,\dots,a\}$, choosing the labels uniformly and independently.
    Then remove all cards labeled 1, keeping them in the same relative order, and
    start a new deck. Remove the cards labeled 2, keeping them in the same relative
    order, and place them below the ones labeled 1. Continue, placing the cards
    labeled $a$ at the bottom. This can be seen as a random walk on the braid
    arrangement.

    More formally, let $a \ge 2$ be an integer and $\aA$ be the braid arrangement in
    $\RR^n$, as before. A \textit{weak ordered partition} of $[n]$ is a sequence of
    pairwise disjoint sets (called blocks) whose union is equal to $[n]$. From such
    a sequence one gets an ordered partition of $[n]$ by simply removing the empty
    blocks. We define a probability measure $w$ on $\fF_\aA$ by first assigning
    weight equal to $1/a^n$ to each of the $a^n$ weak ordered partitions $(B_1,
    B_2,\dots,B_a)$ of $[n]$ with $a$ blocks and then letting $w(B)$ equal the sum
    of the weights of all weak ordered partitions of $[n]$ with $a$ blocks which
    correspond to the ordered partition $B$. For instance, if $a=2$, then
      $$ w(B) \ = \ \begin{cases} 1/2^{n-1}, & \text{if \ $B = ([n])$} \\
      1/2^n, & \text{if \ $B = (s, [n] \sm s)$ and
                                   $s \ne \varnothing, \, s \ne [n]$}\\
                                   0, & \text{otherwise} \end{cases}$$
    for an ordered partition $B$ of $[n]$.

    The resulting chain on $\sss_n$ proceeds from a given permutation $\tau$ by
    selecting uniformly at random a weak ordered
    partition $(B_1, B_2,\dots,B_a)$ of $[n]$ with $a$ blocks and listing the
    elements of each block $B_j$ in the relative order in which they appear in
    $\tau$, to reach a new permutation $\tau'$ (such a permutation is said to be
    obtained from $\tau$ by an \emph{inverse $a$-shuffle}). Equivalently, the
    transition matrix $K$ of the chain satisfies
      $$ K(\tau, \tau') \ = \ \frac{\nu (\tau, \tau') }{a^n},$$
    where $\nu (\tau, \tau')$ is the number of weak ordered partitions of $[n]$
    with $a$ blocks, the projection of $\tau$ on which is equal to $\tau'$.

    Let $W \in \lL_\aA$ be an intersection subspace of codimension $k = \codim(W,
    V)$ and let $\sigma$ be the corresponding partition of $[n]$, so that the
    number of blocks of $\sigma$ is equal to $n-k$. Then the right-hand side of
    (\ref{eq:lambdaW}) is equal to the probability that the following holds for
    a random weak ordered partition $B$ of $[n]$ with $a$ blocks: for every pair
    $\{i, j\}$ of elements of $[n]$ belonging to the same block of $\sigma$, the
    elements $i$ and $j$ also belong to the same block of $B$. This probability
    is clearly equal to $1/a^k$ and hence
      \begin{equation} \label{eq:1/a}
        \lambda_W \ = \ 1/a^{\codim(W, V)}.
      \end{equation}
    Thus it follows easily from Theorem \ref{thm:BHR-BD} (i) (see \cite[Equation
    (31)]{BHR}) that the distinct eigenvalues of $K$ are $1, 1/a,
    1/a^2,\dots,1/a^{n-1}$ and that the multiplicity of the eigenvalue $1/a^i$ is
    equal to the number of permutations in $\sss_n$ which have exactly $n-i$ cycles.
    The stationary distribution $\pi$ in this case is the uniform distribution
    on $\sss_n$ and (\ref{eq:tot2}) gives the bound
      $$ \| K_\tau^l - \pi \|_\TV \ \le \ {n \choose 2} \left( \frac{1}{a}
         \right)^l. $$
    The expression on the right is bounded above by $a^{-c}/2$ if $l \ge 2\log_a n
    +c$ and $c>0$. In fact $(3/2) \log_a n +c$ steps are necessary and sufficient
    for convergence to uniformity; see \cite{BaD} for further details and
    asymptotics.

    The chain of inverse $a$-shuffles converges to the uniform distribution at
    precisely the same rate as the chain of ordinary riffle shuffles on $\sss_n$.
    Thinking of the elements of $\sss_n$ as linear orderings of a deck of $n$
    cards, this chain proceeds from a given ordering as follows. The deck is cut
    into $a$ (possibly empty) packets according to the multinomial distribution
    on their sizes. Then all $a$ packets are riffled together, each time dropping
    a card from one of the $a$ packets with probability proportional to its
    size, to get to a new ordering of the deck. For more information and
    extensive discussions, see \cite{BaD, Di, Ma}. \qed}
  \end{example}

  Examples \ref{ex:tsetlin} and \ref{ex:a} are two of the most interesting cases of
  general hyperplane walks. Other hyperplane arrangements for which the chambers
  are indexed by familiar combinatorial objects and the associated Markov chain
  has a reasonably down to earth description appear in \cite[Section 3]{BD}.
  Further examples where the probabilistic analysis remains to be done can be
  found in \cite{Ath2, PS, Sta2} and \cite[Lecture 5]{Sta3}.

  \subsection{Functions of a Markov chain}
  \label{sec:functions}

  Let $X_0, X_1, X_2,\dots$ be the successive outcomes of a Markov chain on a
  finite state space $\xxx$. Consider a finite set $\yyy$ and a surjective map
  $f: \xxx \to \yyy$. We may think of $\yyy$ as a set partition
  of $\xxx$ and of the map $f$ as the canonical surjection. Thus $f(x)$ is equal to
  the unique block of $\yyy$ which contains $x$, for every $x \in \xxx$. We set
  $Y_i = f(X_i)$ for each index $i$ and refer to $(Y_i)$ as the stochastic process
  (or chain) on the state space $\yyy$ which is \emph{induced from $(X_i)$} by
  the map $f$.

  A function of a Markov chain is usually not Markov. The following lemma
  gives a necessary and sufficient condition for Markovianity in the situation
  described above. We refer the reader to \cite[Sections 6.3-6.4]{KS} for a good
  elementary treatment. For a more sophisticated treatment and references, see
  \cite{RP}.

  \begin{lemma} [Dynkin's Criterion]  \label{lem:markov}
    Let $(X_i)_{\ge 0}$ be a Markov chain on a finite state space $\mathfrak X$
    and let $\yyy$ be a partition of $\xxx$. The chain induced by the canonical
    surjection $f: \xxx \to \yyy$ is Markov for all starting distributions for
    $X_0$ if and only if for any two distinct blocks $B_0, B_1 \in \yyy$, the
    probability $P (X_1 \in B_1~|~X_0 = x_0)$ is constant in $x_0 \in B_0$.
  \end{lemma}

  It is known that if the chain $(X_i)$ is ergodic with stationary distribution
  $\pi$, then the induced chain $(Y_i)$ has a limiting stationary distribution
  $\bar{\pi}$, given by
     \begin{equation} \label{eq:barpi}
       \bar{\pi} (B) \ = \ \sum_{x \in B} \ \pi (x)
     \end{equation}
  and one may inquire about rates of convergence to stationarity (even if the
  induced chain is not Markov). There has been considerable work on convergence
  rates in the situation of Example \ref{ex:tsetlin} (see \cite{BHP, Fi}) and in
  that of riffle shuffling (see \cite{Di1} for a survey and \cite{ADS} for some
  recent developments and references). Further work appears in Sections
  \ref{sec:appA} and \ref{sec:appB}.

  \section{Main results}
  \label{sec:main}

  This section contains our main theoretical contribution. Following a
  suggestion of Uyemura Reyes \cite{Rey}, we show that the process which is
  induced from a hyperplane walk on the set of chambers of a subarrangement is
  a Markov chain which is itself a hyperplane walk, with transition matrix easily
  computable in terms of the original walk (Corollary \ref{cor:sub}). We also
  give a new proof of the description of the eigenvalues of hyperplane walks
  (part (i) of Theorem \ref{thm:BHR-BD}), which uses only basic enumerative
  combinatorics, and a new proof of the basic convergence theorem (part
  (iv) of Theorem \ref{thm:BHR-BD}), which is perhaps more transparent
  than the one given in \cite{BD}.

  Throughout this section, $\aA$ is a hyperplane arrangement in the vector space
  $V = \RR^n$ with set of chambers $\cC_\aA$ and set of faces $\fF_\aA$, $\bB
  \subseteq \aA$ is a subarrangement with set of chambers $\cC_\bB$ and set of
  faces $\fF_\bB$ and $K$ is the transition matrix of the hyperplane walk on
  $\aA$ associated to a given probability measure $w$ on $\fF_\aA$. Our
  starting point is the observation that every chamber $C \in \cC_\aA$ is
  contained in a unique chamber of $\bB$, which we denote by $\overline{C}$.
  Moreover, every chamber of $\bB$ contains at least one chamber of $\aA$. Thus
  there is a surjective map $f: \cC_\aA \to \cC_\bB$ defined by $f(C) =
  \overline{C}$ for $C \in \cC_\aA$ and hence the hyperplane walk on $\aA$
  associated to $w$ induces a stochastic process on the state space $\cC_\bB$, in
  the sense of Section \ref{sec:functions}. The following proposition verifies
  Dynkin's criterion in this situation.

  \begin{proposition} \label{prop:sub}
    Let $D, D' \in \cC_\bB$ be chambers. If $C \in \cC_\aA$ is any chamber with
    $\overline{C} = D$, then the sum
      \begin{equation} \label{eq:KCD0}
      Q (C, D') \ = \sum_{C' \in \cC_\aA: \ \overline{C'}=D'} \ K(C, C')
      \end{equation}
    depends only on $D$ and $D'$ and not on the choice of $C$.
  \end{proposition}
  \begin{proof}
    Replacing $K (C, C')$ by the right-hand side of (\ref{eq:BHRdef}), we find
    that
    \begin{equation} \label{eq:KCD1}
    Q (C, D') \ = \ \sum_{C' \in \cC_\aA: \ \overline{C'}=D'} \ \
    \sum_{F \in \fF_\aA: \, FC=C'} \ w(F) \ \ = \
    \sum_{F \in \fF_\aA: \, \overline{FC}=D'} \ w(F).
    \end{equation}
  Let us denote by $\overline{F}$ the unique face of $\bB$ which contains $F \in
  \fF_\aA$. It is easy to check that $\overline{FC} = \overline{F} \ \overline{C}$
  holds for every $F \in \fF_\aA$. This observation and (\ref{eq:KCD1}) imply that
    \begin{equation} \label{eq:KCD2}
    Q (C, D') \ = \ \sum_{F \in \fF_\aA: \, \overline{F}D=D'} \ w(F).
    \end{equation}
  Clearly, the right-hand side of (\ref{eq:KCD2}) is independent of the choice
  of $C$.
  \end{proof}

  \begin{corollary} \label{cor:sub}
    For every starting distribution on $\cC_\aA$, the stochastic process induced
    on $\cC_\bB$ from the hyperplane walk on $\aA$ associated to $w$ is Markov.
    Moreover, such an induced chain is itself a hyperplane walk on $\bB$, with
    associated probability measure $w^*$ on $\fF_\bB$ defined by
      \begin{equation} \label{eq:w*def}
      w^* (G) \ = \ \sum_{F \in \fF_\aA: \, F \subseteq G} \ w(F)
    \end{equation}
   for $G \in \fF_\bB$.
  \end{corollary}
  \begin{proof}
  The first statement follows from Proposition \ref{prop:sub} and Lemma
  \ref{lem:markov}. The transition matrix $K^*$ of the induced Markov
  chain on $\cC_\bB$ is given by the right-hand side of (\ref{eq:KCD0}), so that
    \begin{equation} \label{eq:KDD1}
    K^* (D, D') \ = \sum_{C' \in \cC_\aA: \ C' \subseteq D'} \ K (C, C')
    \end{equation}
  holds for $D, D' \in \cC_\bB$, where $C \in \cC_\aA$ is any of the chambers of
  $\aA$ contained in $D$. Finally, we note that (\ref{eq:KCD2}) can be rewritten
  as
    \begin{equation} \label{eq:KDD2}
    K^* (D, D') \ = \ \sum_{G \in \fF_\bB: \, GD=D'} \ w^*(G),
    \end{equation}
  where $w^* (G)$ is as in (\ref{eq:w*def}). This proves the second statement
  in the corollary.
  \end{proof}

  The next statement summarizes the main conclusions of our discussion.

  \begin{theorem} \label{cor:main}
    Let $\aA$ be a hyperplane arrangement in $V$ with set of chambers $\cC_\aA$
    and let $w$ be a probability measure on its set of faces $\fF_\aA$. Let
    $\bB \subseteq \aA$ be a subarrangement with set of chambers $\cC_\bB$ and
    set of faces $\fF_\bB$ and let $K^*$ be the transition matrix of
    the Markov chain on $\bB$ induced from the hyperplane walk on $\aA$
    associated to $w$.
     \begin{itemize}
       \item[(i)] The characteristic polynomial of $K^*$ is given by
       $$\det (xI - K^*) \ = \ \prod_{W \in \lL_\bB} \ (x - \lambda_W)^{m^*_W},$$
       where $\lL_\bB$ is the intersection poset of $\bB$, $\lambda_W$ is as
       in (\ref{eq:lambdaW}),
       $$m^*_W \ = \ |\mu_\bB (V, W)| \ = (-1)^{\codim(W, V)} \mu_\bB (V, W)$$
       and $\mu_\bB$ is the M\"obius function of $\lL_\bB$.

       \item[(ii)] The matrix $K^*$ is diagonalizable.
       \item[(iii)] $K^*$ has a unique stationary distribution $\bar{\pi}$
       if and only if for every $H \in \bB$ there exists a face $F \in \fF_\aA$
       such that $F \not\subseteq H$ and $w(F) > 0$. Moreover, if $w$ is
       separating, so that the stationary distribution $\pi$ of the hyperplane
       walk on $\aA$ also exists, then we have
         \begin{equation} \label{eq:pis}
           \bar{\pi}(D) \ = \ \sum_{C \in \cC_\aA: \, C \subseteq D} \ \pi(C)
         \end{equation}
       for every chamber $D \in \cC_\bB$.
       \item[(iv)] Assume that $\bar{\pi}$ exists and let $(K^*_D)^l$
       be the distribution of the induced chain started from the chamber $D
       \in \cC_\bB$ after $l$ steps. Then its total variation distance from
       $\bar{\pi}$ satisfies
         \begin{equation} \label{eq:tot3}
         \| (K^*_D)^l - \bar{\pi} \|_\TV \ \le \ P \left( F_1 F_2 \cdots F_l
         \subseteq \bigcup_{H \in \bB} H \right),
         \end{equation}
       where $(F_1, F_2,\dots)$ consists of independent and identically
       distributed picks from the measure $w$
       on $\fF_\aA$, and
         \begin{equation} \label{eq:tot4}
         \| (K^*_D)^l - \bar{\pi} \|_\TV \ \le \ \sum_{H \in \bB} \ \lambda_H^l.
         \end{equation}
     \end{itemize}
  \end{theorem}
  \begin{proof}
  Let $w^*$ be as in Corollary \ref{cor:sub}. By Corollary \ref{cor:sub} and
  Theorem \ref{thm:BHR-BD}, the characteristic polynomial of $K^*$ is given
  by the expression suggested in part (i), provided that $\lambda_W$ is
  replaced by
    \begin{equation} \label{eq:lambda*W}
      \lambda^*_W \ = \ \sum_{G \in \fF_\bB: \, G \subseteq W} \ w^*(G)
    \end{equation}
  for every $W \in \lL_\bB$. Since every face $G \in \fF_\bB$ is partitioned
  by the faces $F \in \fF_\aA$ contained in $G$, it follows from
  (\ref{eq:w*def}) that the right-hand sides of (\ref{eq:lambdaW}) and
  (\ref{eq:lambda*W}) coincide. Hence we have $\lambda^*_W = \lambda_W$ for
  every $W \in \lL_\bB$ and part (i) follows. The remaining parts are direct
  consequences of the combination of Corollary \ref{cor:sub} with Theorem
  \ref{thm:BHR-BD}.
  \end{proof}

  We now turn to our new proofs of parts (i) and (iv) of Theorem
  \ref{thm:BHR-BD}. The proof of part (i) is motivated by the combinatorial
  method used in \cite{Ath1} to determine the eigenvalues of some matrices,
  with rows and columns indexed by the oriented rooted spanning trees of a
  directed graph. A related argument which proves diagonalizability is given
  in \cite[Section 8.1]{Br1} \cite[Section 3.4]{Br2}. We denote by $\tr (A)$
  the trace of a $p \times p$ matrix $A = (a_{ij})$, so that
    \begin{equation} \label{eq:tr}
      \tr (A^l) \ = \ \sum_{i=1}^p \ \sum_{1 \le i_1,\dots,i_{l-1}
      \le p} a_{ii_1} a_{i_1i_2} \cdots a_{i_{l-1}i}
    \end{equation}
  holds for every positive integer $l$. The method of \cite{Ath1} is based on
  the following elementary lemma.

  \begin{lemma} \label{lem:eigen}
    Let $A = (a_{ij})$ be a $p \times p$ matrix with complex entries
    and let $\lambda_1, \lambda_2,\dots,\lambda_p$ be complex numbers.
    If $\tr (A^l) = \lambda^l_1 + \lambda^l_2 + \cdots + \lambda^l_p$
    holds for every positive integer $l$, then $\lambda_1,
    \lambda_2,\dots,\lambda_p$ are the eigenvalues of $A$.
  \end{lemma}
  \begin{proof}
    We note that $\tr (A^l) = \mu^l_1 + \mu^l_2 + \cdots +
    \mu^l_p$ holds for every positive integer $l$, where $\mu_1,
    \mu_2,\dots,\mu_p$ are the eigenvalues of $A$. It follows from this
    fact, our hypothesis and \cite[Lemma 2.1]{Ath1} that the $\lambda_i$
    are a permutation of the $\mu_j$. This proves the lemma.
  \end{proof}

  \medskip
  \noindent
  \begin{proof}[Proof of Theorem \ref{thm:BHR-BD} (i)]
  By Lemma \ref{lem:eigen}, it suffices to show that
    $$ \tr (K^l) \ = \sum_{W \in \lL_\aA} \ m_W (\lambda_W)^l $$
  holds for every positive integer $l$. Using the definition of $K$,
  we see that for this matrix (\ref{eq:tr}) can be rewritten as
    \begin{equation} \label{eq:f(K,l)}
    \tr (K^l) \ = \ \sum_{C \in \cC_\aA} \ \sum_{F_1 F_2 \cdots
    F_l C = C} \ w(F_1) w(F_2) \cdots w(F_l),
    \end{equation}
  where the inner sum ranges over all sequences $(F_1, F_2,\dots,F_l)$
  of elements of $\fF_\aA$ of length $l$ satisfying $F_1 F_2 \cdots F_l
  C = C$. Next we claim that for every $F \in \fF_\aA$ we have
    \begin{equation} \label{eq:FC}
    \# \{C \in \cC_\aA: FC = C\} \ = \ \sum_{W \in \lL_\aA: \ F \subseteq W}
    \ |\mu_\aA(V, W)| \ \ = \sum_{W \in \lL_\aA: \ F \subseteq W} \ m_W.
    \end{equation}
  Indeed, for a chamber $C \in \cC_\aA$ we have $FC = C$ if and only if $F$
  lies in the closure of $C$. The chambers of $\aA$ with this property are in
  a one to one correspondence with the chambers of the subarrangement of $\aA$
  consisting of those hyperplanes which contain $F$. Thus (\ref{eq:FC}) follows
  from Zaslavsky's formula \cite[Theorem 2.5]{Sta3} \cite{Za} for the number
  of chambers of this subarrangement. Using equations (\ref{eq:f(K,l)}) and
  (\ref{eq:FC}) we find that

  \begin{eqnarray*}
  \tr (K^l) &=& \sum_{C \in \cC_\aA} \ \sum_{F \in \fF_\aA: \, FC = C} \
  \sum_{F_1 F_2 \cdots F_l = F} \ w(F_1) w(F_2) \cdots w(F_l) \\
  & & \\
  &=& \sum_{F \in \fF_\aA} \ \# \{C \in \cC_\aA: F C = C\}
  \sum_{F_1 F_2 \cdots F_l = F} \ w(F_1) w(F_2) \cdots w(F_l) \\
  & & \\
  &=& \sum_{F \in \fF_\aA} \ \, \sum_{W \in \lL_\aA: \ F \subseteq W}
  \ m_W \
  \sum_{F_1 F_2 \cdots F_l = F} \ w(F_1) w(F_2) \cdots w(F_l) \\
  & & \\
  &=& \sum_{W \in \lL_\aA} \ m_W \ \sum_{F_1 F_2 \cdots F_l \subseteq W} \
  w(F_1) w(F_2) \cdots w(F_l) \\
  & & \\
  &=& \sum_{W \in \lL_\aA} \ m_W \ \sum_{F_1 \cup \cdots \cup F_l \subseteq
  W} \ w(F_1) w(F_2) \cdots w(F_l) \\
  & & \\
  &=& \sum_{W \in \lL_\aA} \ m_W \ \left( \sum_{F \subseteq W} \ w(F)
  \right)^l \ = \ \sum_{W \in \lL_\aA} \ m_W (\lambda_W)^l, \\
  \end{eqnarray*}
  as desired.
  \end{proof}

  Theorem \ref{thm:BHR-BD} (iv) is proved in \cite{BD} by considering backward
  iteration. The following coupling proof is perhaps more transparent. For
  background on coupling we refer the reader to \cite[p.~84]{Di} \cite[Chapter
  5]{LPW}. We recall that the probability measure $w$ on the set of faces of $\aA$
  is assumed to be separating. As before, $\cC_\aA$ is the set of chambers of
  $\aA$.

  \medskip
  \noindent
  \begin{proof}[Proof of Theorem \ref{thm:BHR-BD} (iv)]
  Consider two Markov chains $(X_i)$ and $(Y_i)$ evolving on $\cC_\aA$ as
  follows. The first chain starts at $X_0 = C$ and the second starts with $Y_0$
  chosen from the stationary distribution $\pi$. At time $i$ the face $F_i$ is
  chosen from $w$ and is used to upgrade both chains; thus $X_i = F_i X_{i-1}$
  and $Y_i = F_i Y_{i-1}$. Let $T$ be the first time $l$ that the hyperplanes
  of $\aA$ have been separated by $\{F_1, F_2,\dots,F_l\}$, meaning that for
  every $H \in \aA$ there exists an index $1 \le i \le l$ such that $F_i
  \not\subseteq H$. We claim that at this time we have $X_T = Y_T$. It is then
  clear that $X_j = Y_j$ has to hold for all $j \ge T$. Thus the two chains are
  coupled and (\ref{eq:tot1}) follows from the basic coupling inequality
  \cite[p.~84]{Di} \cite[Chapter 5]{LPW}. Since (\ref{eq:tot2}) follows easily
  from (\ref{eq:tot1}) (see \cite[p.~1839]{BD}), it remains to prove the claim.

  Consider any hyperplane $H \in \aA$ and choose an index $1 \le i \le T$ so
  that $F_i \not\subseteq H$. Then both chambers $X_i = F_i X_{i-1}$ and $Y_i =
  F_i Y_{i-1}$ lie in the same open half-space of $V$ determined by $H$ as $F_i$.
  Therefore these chambers lie in the same open half-space of $V$ determined by
  $H$. It follows by induction on $j$ that the same holds for $X_j$ and $Y_j$
  for all $j \ge i$ and thus for $j=T$ as well. We have shown that for every
  $H \in \aA$, the chambers $X_T$ and $Y_T$ lie in the same open half-space of
  $V$ determined by $H$. Clearly any two such chambers must be equal. This
  proves the claim and completes the proof.
  \end{proof}

  \begin{remark} \label{rem:cproof}
    {\rm As was the case in \cite{BD}, the argument in the previous proof
    does not require that faces are chosen independently from the same
    distribution. Any stationary process works as well. Nonstationary choices
    of face weights may be similarly handled. Then there may not be a stationary
    distribution and one needs to study ``merging" \cite{SZ}.}
  \end{remark}

  \section{Applications to hypercube walks}
  \label{sec:appA}

  Throughout this section, $\aA$ stands for the Boolean arrangement in $\RR^N$
  for some $N$, to be specified in each case. Specializing the choice of
  face weights and subarrangement gives a variety of natural examples. Part A
  treats the Ehrenfest urn of statistical mechanics. A spatial process driven by
  neighborhood attacks is studied in Part B. Part C gives a first treatment of
  the acyclic orientations chain (Example \ref{ex:acyclic} in the introduction);
  the results are summarized in Corollary \ref{cor:o}.

  \medskip
  \noindent
  \textbf{A. Ehrenfest Urn}.
  Consider the Markov chain of Example \ref{ex:nearest} with weights $w_i^\varepsilon
  = 1/2n$ for all $\varepsilon \in \{-, +\}$ and $1 \le i \le n$. This is the usual
  nearest neighbor random walk on the $n$-dimensional cube with holding $1/2$, also
  known as Ehrenfests' urn. The transition matrix $K$ has eigenvalues $j/n$ with
  multiplicity ${n \choose j}$, for $0 \le j \le n$, and a uniform stationary
  distribution $\pi$. This walk has a small literature of its own, reviewed in
  \cite[p.~19]{Di} \cite[Section 2.3]{LPW}. As explained there, the mixing time
  is $\frac{1}{2} n \log n$. The slightly less accurate bound
    $$ \| K_x^l - \pi \|_\TV \ \le \ n \left(1 - \frac{1}{n} \right)^l $$
  follows from (\ref{eq:tot-nearest}) and shows that the total variation distance on
  the left is bounded above by $e^{-c}$ if $l \ge n (\log n +c)$. To illustrate the
  speedup possible for a subarrangement walk in this case, consider the subarrangement
  $\bB = \{H_1, H_2,\dots,H_k\}$ of the Boolean arrangement $\aA$ in $\RR^n$. The
  induced walk is a Markov chain on the set $\{-, +\}^k$. Theorem \ref{cor:main}
  implies that its transition matrix $K^*$ has eigenvalues $(n-j)/n$ with multiplicity
  ${k \choose j}$, for $0 \le j \le k$, and a uniform stationary distribution $\bar{\pi}$.
  Equation (\ref{eq:tot4}) gives
   $$ \| (K^*_y)^l - \bar{\pi} \|_\TV \ \le \ k \left(1 - \frac{1}{n} \right)^l $$
  and hence the expression on the left is bounded above by $e^{-c}$ if
  $l \ge n (\log k +c)$.

  \medskip
  \noindent
  \textbf{B. Neighborhood Attacks}. Let $\gG$ be a (finite, undirected) simple graph
  on the node set $[n]$. Each node of $\gG$ is labeled with either $+$ or $-$. A
  Markov chain on the set $\{-, +\}^n$ of all $2^n$ possible labelings proceeds as
  follows. At each stage, a node of $\gG$ is chosen uniformly at random. The labels
  of this node and of its neighbors are all changed to $+$ or all changed to $-$,
  with probability $1/2$. The transition matrix $K$ for this chain satisfies
    $$K (x, x') \ = \ \frac{\mu(x, x')}{2n} $$
  for $x, x' \in \{-, +\}^n$, where $\mu(x, x')$ is the number of pairs $(i,
  \varepsilon)$ of nodes $i \in [n]$ and signs $\varepsilon \in \{-, +\}$ for which
  $x'$ is obtained from $x$ by changing the labels of $i$ and its neighbors in $\gG$
  to $\varepsilon$. Clearly, this is the chain defined by the hyperplane walk on the
  Boolean arrangement $\aA$ in $\RR^n$ for the following choice of face weights. For
  each node $i \in [n]$ and $\varepsilon \in \{-, +\}$ we denote by $F^\varepsilon_i$
  the face of $\aA$ whose sign vector has $j$-coordinate equal to $\varepsilon$,
  if $j$ is a neighbor of $i$ in $\gG$ or $j=i$, and equal to 0 otherwise. We define
  $w(F)$ as $1/2n$ times the number of pairs  $(i, \varepsilon)$ of nodes $i \in [n]$
  and signs $\varepsilon \in \{-, +\}$ for which $F^\varepsilon_i = F$ (note that we
  may have $F^\varepsilon_i = F^\varepsilon_j$ for distinct nodes $i, j \in [n]$).
  Ehrenfests' urn occurs as the special case in which $\gG$ has no edges.

  For $S \subseteq [n]$ we denote by $\alpha(S)$ the number of nodes of $\gG$ which
  are neither  equal nor adjacent to any of the nodes in $S$. It follows from
  (\ref{eq:lambdaB}) that $K$ has eigenvalues $j/n$, with multiplicity equal to the
  number of subsets $S \subseteq [n]$ with
  $\alpha(S) = j$, for $0 \le j \le n$ and that for $1 \le i \le n$, the eigenvalue
  contributed by the hyperplane $H_i$ of $\aA$ is equal to $1 - (d_i + 1)/n$, where
  $d_i$ is the degree of node $i$ in $\gG$. The stationary distribution $\pi$
  for this example exists for every graph $\gG$ but is hard to compute in general.
  Inequality (\ref{eq:tot1}) bounds the total variation distance $\| K_x^l - \pi
  \|_\TV$ from above by the probability that
    $$ \bigcup_{i=1}^l \ N(v_i) \ \ne \ [n], $$
  where nodes $v_1, v_2,\dots,v_l$ are picked independently and uniformly from $[n]$
  and $N(v)$ stands for the set of nodes of $\gG$ which are either adjacent or
  equal to $v$. To compute this probability is a classical problem, called the
  ``coverage problem"; see, for instance, \cite{Al, BHJ, CKS}. Similarly, the
  eigenvalue bound (\ref{eq:tot2}) becomes
     \begin{equation} \label{eq:tot-attack}
       \| K_x^l - \pi \|_\TV \ \le \ \sum_{i=1}^n \ \left(1 - \frac{d_i+1}{n}
       \right)^l.
     \end{equation}
  For instance, if $\gG$ is the complete graph on the node set $[n]$, then $d_i =
  n-1$ for all $i$ and the walk becomes random after exactly one step.

  The eigenvalue bound is not perfect. For instance, consider a ``star graph",
  having one central node of degree $n-1$, and $n-1$ leaves of degree one. The
  right-hand side of (\ref{eq:tot-attack})
  becomes $(n-1)(1-2/n)^l$ and shows that order of $n \log n$ steps suffice. On the
  other hand, the coverage bound is bounded above by $(1-1/n)^l$, which is the chance
  of missing the central node in the first $l$ steps. This implies that order of $n$
  steps suffice. An elementary argument shows that this is the correct answer. For a
  general graph $\gG$, (\ref{eq:tot-attack}) implies that $\| K_x^l - \pi \|_\TV \le
  e^{-c}$ if $l \ge \frac{n}{d+1} (\log n + c)$, where $d$ is the largest of the
  degrees $d_i$.

  To estimate the time it takes for a subset of nodes, say $\{1, 2,\dots,k\}$, to
  equilibriate, consider the subarrangement $\bB = \{H_1, H_2,\dots,H_k\}$ of $\aA$
  and note that (\ref{eq:tot4}) becomes
    $$ \| (K^*_y)^l - \bar{\pi} \|_\TV \ \le \ \sum_{i=1}^k \
       \left(1 - \frac{d_i+1}{n} \right)^l. $$
  This offers a range of possibilities to illustrate the speedup possible; we leave
  further details and examples to the interested reader. One can also deduce easily
  from Theorem \ref{cor:main} that the transition matrix $K^*$ of the induced chain
  has eigenvalues $\alpha(S)/n$, one for each $S \subseteq [k]$.

  The previous situation can be easily varied; the nodes can be chosen with varying
  probability, the size and shape of the neighborhood can be allowed to fluctuate and
  the chance of $+$ or $-$ need not be symmetric. With such freedom, the stationary
  distribution becomes intractable but it is still staightforward to give upper bounds
  for the total variation distance to stationarity. Lower bounds are harder to achieve,
  due to our lack of knowledge of the stationary distribution.


  \medskip
  \noindent
  \textbf{C. Orientations}. Let $\gG$ be a (finite, undirected) simple graph on the
  node set $[n]$ with $m$ edges. An \emph{orientation} of $\gG$ is an assignment
  of a direction $i \to j$ or $j \to i$ to each edge $\{i, j\}$ of $\gG$. We will
  denote by ${\rm O} (\gG)$ the set of all orientations of $\gG$. This set is in
  bijection with $\{-, +\}^m$ and hence with the set of chambers of the Boolean
  arrangement $\aA$ in $\RR^m$. To be more specific, let $E_\gG = \{e_1,
  e_2,\dots,e_m\}$ be the set of edges of $\gG$, equipped with a fixed linear
  ordering of its elements, and let us identify an orientation $\oo \in {\rm O}
  (\gG)$ with the sign vector $(\varepsilon_1, \varepsilon_2,\dots,\varepsilon_m)
  \in \{-, +\}^m$ for which
    $$ \varepsilon_k \ = \ \begin{cases} -, & \text{if \ $e_k$ is directed as $i
                         \to j$ in $\oo$ and $i<j$} \\
                         +, & \text{if \ $e_k$ is directed as $i
                         \to j$ in $\oo$ and $i>j$} \end{cases}$$
  for $1 \le k \le m$, where $e_k = \{i , j\}$. Thus any hyperplane walk on $\aA$
  defines a Markov chain on ${\rm O} (\gG)$. A choice of face weights which gives
  Example \ref{ex:acyclic} of the introduction is
  the following. Let $w_1, w_2,\dots,w_n$ be nonnegative real numbers summing to
  1. For $1 \le i \le n$, we denote by $F_i$ the face of $\aA$ whose sign
  vector has $k$th coordinate equal to $+$, if $e_k = \{i, j\}$ with $i<j$, to $-$,
  if $e_k = \{i, j\}$ with $i>j$ and to 0, if $e_k$ is not incident to $i$. We let
  $w(F_i) = w_i$ for each node $i \in [n]$ which is not isolated in $\gG$ and $w(F)
  = 0$ for all other nonzero faces of $\aA$, and define $w(F)$ as the sum of $w_i$
  over all isolated nodes $i \in [n]$ of $\gG$, if $F$ is the zero face of $\aA$.

  The resulting chain on ${\rm O} (\gG)$ proceeds from a given orientation
  by selecting the node $i$ of $\gG$ with probability $w_i$ and reorienting all
  edges of $\gG$ incident to this node towards itself, to reach a new orientation
  of $\gG$, leaving the orientations of all other edges of $\gG$ unchanged.
  Equivalently, the transition matrix $K$ of this chain on ${\rm O} (\gG)$
  satisfies
    \begin{equation} \label{eq:otransition}
       K (\oo, \oo') \ = \ \sum_{i \in [n]: \, \rho_i (\oo)=\oo'} \ w_i
    \end{equation}
  for $\oo, \oo' \in {\rm O} (\gG)$, where $\rho_i (\oo)$ denotes the
  orientation of $\gG$ obtained from $\oo$ by reorienting towards $i$, as just
  described. We collect the consequences of Theorem
  \ref{thm:BHR-BD} for this example in the following statement. We denote the
  stationary distribution by $\Pi$ to avoid confusion with the notation of Section
  \ref{sec:appB}, where acyclic orientations of $\gG$ are considered and $\pi$
  has a different meaning. A subset $T$ of the set of nodes of $\gG$ is said to be
  \emph{dominating} in $\gG$ if every edge of $\gG$ is incident to at least one
  node in $T$.

  \begin{corollary} \label{cor:o}
  Let $\gG$ be a simple graph on the node set $[n]$ and let $E_\gG = \{e_1,
  e_2,\dots,e_m\}$ be the set of edges of $\gG$. The following hold for the chain
  (\ref{eq:otransition}) on the set ${\rm O} (\gG)$ of orientations of $\gG$:
     \begin{itemize}
       \item[(i)] The matrix $K$ is diagonalizable with eigenvalues
         \begin{equation} \label{eq:lambdaO}
         \lambda_S \ = \ \sum_{i \in N_S} \ w_i,
         \end{equation}
       one for each $S \subseteq [m]$, where $N_S$ is the set of nodes $i \in [n]$
       which do not belong to any of the edges $e_k \in E_\gG$ with $k \in S$.
       \item[(ii)] $K$ has a unique stationary distribution $\Pi$ if and only if there
       is no edge $\{i, j\} \in E_\gG$ such that $w_i = w_j = 0$.
       \item[(iii)] Assume that $\Pi$ exists and let $K_\oo^l$ be the distribution
       of the chain started from the orientation $\oo \in {\rm O} (\gG)$ after $l$
       steps. We have
         \begin{equation} \label{eq:otot1}
         \| K_\oo^l - \Pi \|_\TV \ \le \ P \left( \{v_1, v_2,\dots,v_l\}
         \ \text{is not dominating in} \ \gG \right),
         \end{equation}
       where $(v_1, v_2,\dots)$ consists of independent and identically
       distributed picks from $w$, and
         \begin{equation} \label{eq:otot2}
         \| K_\oo^l - \Pi \|_\TV \ \le \ \sum_{\{i, j\} \in E_\gG} \
         (1 - w_i - w_j)^l.
         \end{equation}
       In particular, we have
         \begin{equation} \label{eq:otot3}
         \| K_\oo^l - \Pi \|_\TV \ \le \ m \left(1 - \frac{2}{n}
         \right)^l
         \end{equation}
       if $w_1 = \cdots = w_n = 1/n$.
     \end{itemize}
  \end{corollary}
  \begin{proof}
  Parts (i) and (ii) follow directly from Theorem \ref{thm:BHR-BD} and the relevant
  discussion in Section \ref{sec:ex}. Part (iii) follows from Theorem \ref{thm:BHR-BD}
  (iv), since a product of faces of $\aA$ of the form $F_v$ is a chamber if and only
  if the corresponding set of nodes $v$ is dominating in $\gG$ and since
  $\lambda_{\{e_k\}} = 1 - w_i - w_j$ is the eigenvalue corresponding to the hyperplane
  $x_k = 0$ of $\aA$ associated to the edge $e_k = \{i, j\}$ of $\gG$.
  \end{proof}

  \begin{example}
    {\rm Let $n = 2d$ be even and consider the graph $\gG$ with edges $\{1, d+1\}$,
    $\{2, d+2\},\dots,\{d, 2d\}$. The set of orientations of $\gG$ can be identified
    with $\{-, +\}^d$, where the $i$th coordinate of a sign vector is equal to $+$
    or $-$ if the edge $\{i, d+i\}$ is directed towards $i$ or towards $d+i$,
    respectively, in the corresponding orientation. The chain proceeds, at each stage,
    from the current sign vector by picking a coordinate $i$ and switching it to $+$
    (respectively, $-$) with probability $w_i$ (respectively, $w_{i+d}$). Clearly,
    this chain concides with the nearest neighbor random walk of Example
    \ref{ex:nearest} on the vertex set of the $d$-dimensional cube, where $w_i$ and
    $w_{i+d}$ have the roles played by $w^+_i$ and $w^-_i$, respectively, in
    that example.
    \qed}
  \end{example}

  We postpone the description of the stationary distribution for a general graph
  $\gG$ until Section \ref{sec:appB} (see Proposition \ref{prop:Pi}), where
  more examples also appear.

  \section{Applications to permutation walks}
  \label{sec:appB}

  Throughout this section, $\aA$ stands for the braid arrangement in $\RR^n$.
  A subarrangement of $\aA$ is specified by a simple graph $\gG$ on the node set
  $[n]$. It is first shown that every hyperplane walk on $\aA$ induces a walk on
  the set of acyclic orientations of $\gG$ (Proposition \ref{prop:ao}).
  Specializing to the Tsetlin library walk in Part A gives again the walk on
  acyclic orientations of Example \ref{ex:acyclic}. We give a detailed discussion,
  determining the eigenvalues, stationary distribution and rates of convergence.
  A birth and extinction example shows that the coupling bound (\ref{eq:tot3})
  can be much better than the eigenvalue bound (\ref{eq:tot4}). Part B shows how
  various aspects of a permutation behave after successive riffle shuffles. This
  yields a probabilistic interpretation for the coefficients of the chromatic
  polynomial of a graph. Descents of permutations of various types are also
  treated.

  \medskip
  \noindent
  {\bf Acyclic Orientations.}
  Let $\gG$ be a simple (undirected) graph on the node set $[n]$.
  Given an orientation $\oo$ of $\gG$, a \emph{directed cycle} in $\oo$ is a
  sequence of nodes $(i_0, i_1,\dots,i_k)$
  such that $i_0 \to i_1 \to \cdots \to i_k \to i_0$ in $\oo$. An orientation
  of $\gG$ is \emph{acyclic} if it contains no directed cycles. An acyclic
  orientation $\oo$ of $\gG$ specifies a partial order $\preceq_\oo$ on the
  set $[n]$ by letting $a \preceq_\oo b$ if there exists a directed walk in
  $\oo$ with initial node $a$ and final node $b$. Conversely, every partial
  order $\preceq$ on $[n]$ comes from a graph on the node set $[n]$ in this
  way (for instance, from the Hasse diagram of $\preceq$).

  The acyclic orientations of $\gG$ can be modeled by the chambers of a
  hyperplane arrangement as follows. The \emph{graphical arrangement}
  corresponding to $\gG$ is the subarrangement $\bB$ of the braid arrangement
  $\aA$ in $\RR^n$, consisting of all hyperplanes of the form $x_i - x_j = 0$
  for which $\{i, j\}$ is an edge
  of $\gG$. A chamber $D \in \cC_\bB$ defines an acyclic orientation of $\gG$
  by assigning the direction $j \to i$ to the edge $\{i, j\}$ of $\gG$ if $x_i
  > x_j$ holds for every point $(x_1, x_2,\dots,x_n) \in D$. The resulting
  map is a bijection from the set of chambers $\cC_\bB$ to the set of acyclic
  orientations of $\gG$, henceforth denoted by $\AO (\gG)$; see \cite[Section
  2.3]{Sta3} for a proof and further information. Using this bijection, we may
  identify chambers of $\bB$ with the corresponding acyclic orientations of
  $\gG$.

  It follows from the previous discussion that every hyperplane walk on $\aA$
  induces a Markov chain on the set $\AO (\gG)$, as described in Section
  \ref{sec:main}. We record this conclusion in the following proposition. We
  recall that a permutation $\tau \in \sss_n$ is said to be a \emph{linear
  extention} of a partial order $\preceq$ on $[n]$, if for all $a, b \in [n]$
  with $a \prec b$ we have $\tau^{-1}(a) < \tau^{-1}(b)$, meaning that $a$
  appears before $b$ in the linear ordering $(\tau(1), \tau(2),\dots,\tau(n))$
  associated to $\tau$.

  \begin{proposition} \label{prop:ao}
    Every hyperplane walk on the braid arrangement in $\RR^n$ induces a Markov
    chain on the set $\AO (\gG)$ of acyclic orientations of $\gG$. If the
    original walk has a unique stationary distribution $\pi$, then the
    stationary distribution $\bar{\pi}$ of the induced chain is given by
      \begin{equation} \label{eq:aopi}
        \bar{\pi}(\oo) \ = \ \sum_{\tau \in \eE (\oo)} \ \pi(\tau)
      \end{equation}
    for $\oo \in \AO(\gG)$, where $\eE (\oo)$ is the set of linear extensions
    of the partial order on $[n]$ defined by $\oo$. In particular, if $\pi$
    is the uniform distribution on $\sss_n$, then
      \begin{equation} \label{eq:aoupi}
        \bar{\pi}(\oo) \ = \ \frac{\# \eE (\oo)}{n!}
      \end{equation}
    for every $\oo \in \AO(\gG)$.
  \end{proposition}
  \begin{proof}
  Let $\aA$ denote the braid arrangement in $\RR^n$ and $\bB$ denote the
  graphical arrangement corresponding to $\gG$, as before. As already mentioned,
  the first statement follows from the previous discussion and Corollary
  \ref{cor:sub}. The second statement follows from (\ref{eq:pis}) and the
  observation that for chambers $C \in \cC_\aA$ and $D \in \cC_\bB$ corresponding
  to the permutation $\tau \in \sss_n$ and the acyclic orientation $\oo \in \AO
  (\gG)$, respectively, we  have $C \subseteq D$ if and only if $\tau \in \eE
  (\oo)$.
  \end{proof}

  \begin{remark} \label{rem:ao}
  {\rm Acyclic orientations are of importance in various areas of applied
  mathematics, such as computer science, automata theory and statistics. In
  statistical applications they appear as part of the machinery of ``Bayes
  nets" and ``casual models", where they are used to model casual implication
  in complex data sets; some useful references are \cite{Fr, La, GP}.
  Searching for an appropriate model is often done by a random walk on acyclic
  orientations. We hope that our analysis will contribute to the understanding
  of these algorithms. For an introduction to the literature relating acyclic
  orientations to factoring noncommutative polynomials, see \cite{RW}. \qed}
  \end{remark}

  In the remainder of this section we investigate further the Markov chain of
  Proposition \ref{prop:ao} in the special cases of Examples \ref{ex:tsetlin}
  and \ref{ex:a}. Note that the case of Example \ref{ex:tsetlin} is also treated
  by the Markov chain (\ref{eq:otransition}).

  \medskip
  \noindent
  \textbf{A. Tsetlin Library}.
  Let $w_1, w_2,\dots,w_n$ be nonnegative real numbers summing to 1 and let $w$
  be the probability measure on $\fF_\aA$ of Example \ref{ex:tsetlin}. Thus
  the associated hyperplane walk on $\aA$ is the Markov chain on $\sss_n$ which
  selects, at each stage, the entry $i$ in the one line notation of the current
  permutation with
  probability $w_i$ and moves it in front. To describe the induced chain of
  Proposition \ref{prop:ao} on the set $\AO (\gG)$, we observe the following:
  if $C \in \cC_\aA$ is the chamber which corresponds to a given permutation
  $\tau \in \sss_n$ and $D \in \cC_\bB$ is the unique chamber of $\bB$ which
  contains $C$, then the acyclic orientation of $\gG$ corresponding to $D$ is
  the one which orients an edge $\{a, b\}$ of $\gG$ as $b \to a$ if and only if
  $\tau^{-1}(a) < \tau^{-1}(b)$. It follows that the induced chain on $\AO (\gG)$
  proceeds from a given acyclic orientation by selecting the node $i$ of $\gG$
  with probability $w_i$ and reorienting all edges of $\gG$ incident to this
  node towards itself, to reach a new acyclic orientation of $\gG$, leaving the
  orientations of all other edges of $\gG$ unchanged. Equivalently, if $K^*$ is
  the transition matrix of the induced chain on $\AO (\gG)$, then $K^* (\oo,
  \oo')$ is given by the right-hand side of (\ref{eq:otransition})
  for $\oo, \oo' \in \AO (\gG)$.

  Given a subset $T$ of the node set of $\gG$, we denote by $\gG \sm T$ the graph
  obtained from $\gG$ by removing all nodes in $T$ and all incident to them edges
  (in other words, $\gG \sm T$ is the induced subgraph of $\gG$ on the node set
  $[n] \sm T$). Thus $T$ is dominating in $\gG$ if and only if the graph $\gG \sm
  T$ has no edges. The following statements can be added to the conclusions of
  Proposition \ref{prop:ao}.

  \begin{proposition} \label{prop:tsetlin}
    Let $\gG$ be a simple graph on the node set $[n]$ and let $K^*$ be the
    transition matrix of the Markov chain on $\AO (\gG)$ which is induced from the
    Tsetlin library with weights $w_1, w_2,\dots,w_n$.
     \begin{itemize}
       \item[(i)] The matrix $K^*$ is diagonalizable with characteristic
       polynomial given by
         \begin{equation} \label{eq:charAO}
         \det (xI - K^*) \ = \ \prod_{S \subseteq [n]} \ (x - \lambda_S)^{m_S},
         \end{equation}
       where
         \begin{equation} \label{eq:lambda_S}
         \lambda_S \ = \ \sum_{i \in S} \ w_i
         \end{equation}
       and
         \begin{equation} \label{eq:m_S}
         m_S \ = \ \sum_{S \subseteq T \subseteq [n]} \ (-1)^{|T \sm S|} \ \#
               \AO (\gG \sm T)
         \end{equation}
       for $S \subseteq [n]$, where the number of acyclic orientations of the graph
       with empty node set is equal to one, by convention.
       \item[(ii)] $K^*$ has a unique stationary distribution $\bar{\pi}$ if
       and only if there is no edge $\{i, j\}$ of $\gG$ such that $w_i = w_j = 0$.
       Moroever, we have
       $$ \bar{\pi}(\oo) \ = \ \sum_{\tau \in \eE (\oo)} \
         \frac{w_{\tau(1)}w_{\tau(2)} \cdots w_{\tau(n)}}
         {(1 - w_{\tau(1)}) (1 - w_{\tau(1)} - w_{\tau(2)}) \cdots
         (1 - w_{\tau(1)} - \cdots - w_{\tau(n-1)}) } $$
       for $\oo \in \AO (\gG)$, if $w_1, w_2,\dots,w_n$ are all positive. In
       particular, (\ref{eq:aoupi}) holds for every $\oo \in \AO(\gG)$ if $w_1
       = \cdots = w_n = 1/n$.
       \item[(iii)] Assuming that $\bar{\pi}$ exists, the conclusions of Corollary
       \ref{cor:o} (iii) hold if $\pi$ is replaced there by $\bar{\pi}$ and
       $K_\oo^l$ is replaced by the distribution $(K^*_\oo)^l$ of the induced chain
       started from the orientation $\oo \in \AO (\gG)$ after $l$ steps.
     \end{itemize}
  \end{proposition}
  \begin{proof}
  We first recall from \cite[Section 2.3]{Sta3} the following description of
  the intersection poset $\lL_\bB$ of the graphical arrangement $\bB \subseteq
  \aA$ corresponding to $\gG$. A set partition $\pi$ of $[n]$ is said to be
  \emph{$\gG$-connected} if the induced subgraph of $\gG$ on each block of
  $\pi$ is connected. We denote by $\lL_\gG$ the set of $\gG$-connected partitions
  of $\gG$, ordered by refinement. This poset, known as the ``bond lattice", or
  ``lattice of contractions" of $\gG$, is isomorphic to $\lL_\bB$, where the
  isomorphism is induced by the one between the lattice of partitions of $[n]$
  and $\lL_\aA$, discussed in Section \ref{sec:ex}.

  Given $W \in \lL_\bB$, we write $\lambda_\sigma$ for the eigenvalue $\lambda_W$
  of $K^*$ which appears in Theorem \ref{cor:main}, where $\sigma \in \lL_\gG$
  is the $\gG$-connected partition corresponding to $W$. The definition of the
  measure $w$ on $\fF_\aA$ of Example \ref{ex:tsetlin} and the definition of
  $\lambda_W$ in (\ref{eq:lambdaW}) imply that
    \begin{equation} \label{eq:lamdasigma}
      \lambda_\sigma \ = \ \sum_{\{i\} \in \sigma} \ w_i,
    \end{equation}
  where the sum runs over all singleton blocks $\{i\}$ of $\sigma \in \lL_\gG$.
  Theorem \ref{cor:main} (i) gives
    $$\det (xI - K^*) \ = \ \prod_{\sigma \in \lL_\gG} \ (x -
      \lambda_\sigma)^{m^*_\sigma},$$
  where $m^*_\sigma = |\mu_\gG (\hat{0}, \sigma)|$ and $\mu_\gG$ is the M\"obius
  function of $\lL_\gG$. The previous two equations imply that (\ref{eq:charAO})
  holds if we define
    $$ m_S \ = \ \sum_{\sigma \in \lL_\gG: \ \sing(\sigma) = S} \
    |\mu_\gG (\hat{0}, \sigma)|, $$
  where $\sing(\sigma) = \{i \in [n]: \{i\} \in \sigma\}$ denotes the set of
  singleton blocks of $\sigma$. To complete the proof of part (i), it remains to
  prove (\ref{eq:m_S}). By inclusion-exclusion we can write
    \begin{equation} \label{eq:m_S2}
      m_S \ = \ \sum_{S \subseteq T \subseteq [n]} \ (-1)^{|T \sm S|} \ n_T,
    \end{equation}
  where
    $$ n_T \ = \ \sum_{\sigma \in \lL_\gG: \  T \subseteq \sing(\sigma)} \
    |\mu_\gG (\hat{0}, \sigma)|. $$
  Clearly, writing $\sigma = \{B_1, B_2,\dots,B_k\}$, the closed interval $[\hat{0},
  \sigma]$ in $\lL_\gG$ is isomorphic to the direct product of the lattices
  $\lL_{\gG_i}$ for $1 \le i \le k$, where $\gG_i$ is the induced subgraph of $\gG$
  on the node set $B_i$. It follows easily from this observation and the
  multiplicativity of the M\"obius function \cite[Proposition 3.8.2]{StaV1} that
    \begin{equation} \label{eq:n_T}
      n_T \ = \ \sum_{\sigma \in \lL_{\gG \sm T}} \ |\mu_{\gG \sm T} (\hat{0},
                \sigma)|.
    \end{equation}
  By Zaslavsky's formula \cite[Theorem 2.5]{Sta3} \cite{Za}, the right-hand side of
  (\ref{eq:n_T}) is equal to the number of chambers of the graphical arrangement
  corresponding to $\gG \sm T$ and hence to the number of acyclic orientations of
  $\gG \sm T$. Thus (\ref{eq:m_S}) follows from (\ref{eq:m_S2}) and (\ref{eq:n_T}).

  Part (ii) follows from Theorem \ref{cor:main} (iii), equation (\ref{eq:luce})
  and Proposition \ref{prop:ao}.

  Part (iii) follows from Theorem \ref{cor:main} (iv), since a product of faces of
  $\aA$ corresponding to ordered partitions of the form $(i, [n] \sm \{i\})$ is not
  contained in any of the hyperplanes of $\bB$ if and only if the corresponding set
  of nodes $v_i$ is dominating in $\gG$ and since $\lambda_H = 1 - w_i - w_j$ holds
  for the hyperplane $H$ of $\bB$ corresponding to the edge $\{i, j\}$ of $\gG$.
  \end{proof}

  Part (ii) of Proposition \ref{prop:tsetlin} and the following statement determine
  the stationary distribution of the Markov chain on the set ${\rm O} (\gG)$ of all
  orientations of $\gG$, discussed in part C of Section \ref{sec:appA}. We note that
  if $\gG$ is a forest, then every orientation of $\gG$ is acyclic and hence the two
  Markov chains on ${\rm O} (\gG)$ and $\AO (\gG)$ coincide.

  \begin{proposition} \label{prop:Pi}
  Consider the chain on the set ${\rm O} (\gG)$ of all orientations and the chain on
  the set $\AO (\gG)$ of acyclic orientations of $\gG$, with weights $w_1,
  w_2,\dots,w_n$. Assuming there is no edge $\{i, j\}$ of $\gG$ such that $w_i = w_j
  = 0$, their respective stationary distributions $\Pi$ and $\bar{\pi}$ are related by
    \begin{equation} \label{eq:Pipi}
      \Pi (\oo) \ = \ \begin{cases} \bar{\pi} (\oo), & \text{if $\oo$ is acyclic}\\
                                     0, & \text{otherwise} \end{cases}
    \end{equation}
  for $\oo \in {\rm O} (\gG)$.
  \end{proposition}
  \begin{proof}
  We denote by $K$ the transition matrix of the chain on ${\rm O} (\gG)$ and recall
  that $\Pi (\oo) = \lim_{l \to \infty} K^l(\oo_*, \oo)$ for $\oo \in {\rm O} (\gG)$,
  where the limit is independent of the starting orientation $\oo_* \in {\rm O} (\gG)$.
  Choosing $\oo_* \in \AO (\gG)$, all orientations in the chain stay in $\AO (\gG)$
  and the limit becomes equal to the right-hand side of (\ref{eq:Pipi}).
  \end{proof}

  \begin{example} \label{ex:aoK}
  {\rm Fix an integer $1 \le k \le n$ and let $\gG$ be the graph with edges $\{i, j\}$
  for $1 \le i < j \le k$. The set $\AO (\gG)$ can be identified with the subgroup
  $\sss_k$ of permutations in $\sss_n$ which fix the set $\{k+1,\dots,n\}$ pointwise
  and the induced chain is the process which records the relative ordering of $\{1,
  2,\dots,k\}$, when $\tau \in \sss_n$ evolves as in the Markov chain of Example
  \ref{ex:tsetlin}. The eigenvalues and stationary distribution $\bar{\pi}$ of the
  transition matrix $K^*$ can be easily deduced from those of the transition matrix $K$
  of the parent chain, since in this case $K^*$ differs by a multiple of the identity
  matrix from the restriction of $K$ on $\sss_k$. For instance, $K^*$ has eigenvalues
  (\ref{eq:phat}), one for each $\tau \in \sss_k$. By Proposition \ref{prop:tsetlin}
  (iii), the bound from (\ref{eq:otot2}) applies and gives
    $$ \| (K^*_\tau)^l - \bar{\pi} \|_\TV \ \le \ \sum_{1 \le i < j \le k} \
         (1 - w_i - w_j)^l. $$
  In particular, if $w_1 = \cdots = w_k = 1/k$, so that $\bar{\pi}$ is uniform, then
    \begin{equation} \label{eq:tse4}
      \| (K^*_\tau)^l - \bar{\pi} \|_\TV \ \le \ {k \choose 2} \left(1 - \frac{2}{n}
         \right)^l,
    \end{equation}
  which is an improvement over (\ref{eq:tse2}). Thus the left-hand side of (\ref{eq:tse4})
  is bounded above by $e^{-2c}/2$ if $l \ge n (\log k + c)$.

  The bound (\ref{eq:tse4}) is quite sharp across the whole range of $k$. For instance,
  if $k=2$ it shows that $l$ must grow as $cn$, with $c$ approaching infinity. This is
  correct since if $c$ stays bounded, then there is a nonzero chance that neither 1 nor
  2 has been moved, and thus that they have stayed in their original relative order. At
  the other extreme, we have already commented in our discussion of (\ref{eq:tse2})
  that the bound is sharp if $k=n$. Similar remarks hold for other values of $k$.
  \qed}
  \end{example}

  The following example gives a concrete case in which the bound of (\ref{eq:tot3})
  is better than that of (\ref{eq:tot4}).

  \begin{example} [Birth and Extinction]  \label{ex:aostar}
    {\rm Let $\gG$ be the graph with edges $\{i, n\}$ for $1 \le i \le m$, where
    $m \le n-1$ is a positive integer, and choose weights $w_1 = \cdots =
    w_n = 1/n$. Every orientation of $\gG$ is acyclic and hence the set $\AO (\gG)$
    can be identified with $\{-, +\}^m$, as described in part C of Section
    \ref{sec:appA}. The chain proceeds, at each stage, from a sign vector $x
    \in \{-, +\}^m$ by picking a coordinate $i$ uniformly at random and switching
    this coordinate to $+$, if $1 \le i \le m$, leaving $x$ unchanged, if $m+1 \le
    i < n$, and switching all coordinates of $x$ to $-$, if $i=n$, to reach a new
    sign vector. Such processes are studied in mathematical genetics with many
    variations.

    From our current point of view, we may think of this chain as the process which
    records the subset of $[m]$ consisting of those integers which precede $n$ in the
    current permutation $\tau$, when $\tau \in \sss_n$ evolves as in the Markov chain
    of Example \ref{ex:tsetlin} with uniform weights (random to top model). This is
    because a number $i \in [m]$ precedes $n$ in some (equivalently, every) linear
    extension of the orientation $\oo$ if and only if the edge $\{i, n\}$ is directed
    as $n \to i$ in $\oo$.

  \begin{proposition} \label{prop:aostar}
    Consider the Markov chain on $\AO (\gG)$, which is induced from the Tsetlin library
    with uniform weights, as a chain on the set $\{-, +\}^m$ and let $K^*$ be its
    transition matrix.
     \begin{itemize}
       \item[(i)] The matrix $K^*$ is diagonalizable with eigenvalues
         $$ \begin{cases} 1, & \text{with multiplicity one}\\
             \frac{n-j-1}{n}, & \text{with multiplicity ${m \choose j}$, for
                               $1 \le j \le m$.} \end{cases} $$
       \item[(ii)] The stationary distribution of $K^*$ is given by
         \begin{equation} \label{eq:pix}
           \bar{\pi} (x) \ = \ \frac{1}{(m+1) {m \choose k}}
         \end{equation}
       for $x \in \{-, +\}^m$, where $k$ is the number of coordinates of $x$ equal
       to $+$.
       \item[(iii)] Assume $m = n-1$ and let $(K^*_x)^l$ be the distribution of the
       chain started from $x$ after $l$ steps. We have
         \begin{equation} \label{eq:ostar-tot}
         \| (K^*_x)^l - \bar{\pi} \|_\TV \ \le \ \left(1 - \frac{1}{n} \right)^l \
         \le \ e^{-c}
         \end{equation}
       for $l \ge cn$ and $c>0$. Moreover this bound is sharp, in the sense that
       there exists $0 < \theta < 1$ such that $\| (K^*_x)^n - \bar{\pi} \|_\TV \ge
       \theta$ for all large $n$.
     \end{itemize}
  \end{proposition}
  \begin{proof}
  Part (i) follows from Corollary \ref{cor:o} (i). Alternatively, it follows from
  the proof of Proposition \ref{prop:tsetlin} (i) and, in particular, equation
  (\ref{eq:lamdasigma}), since all values of the M\"obius function $\mu_\gG$ in this
  case have absolute value 1. Part (ii) follows from Proposition \ref{prop:tsetlin}
  (ii) and equation (\ref{eq:aoupi}), which applies in our situation, since the
  number of linear extensions of the poset on $[n]$ defined by any orientation of
  $\gG$ with $k$ edges pointing away from $n$ is equal to ${n \choose m+1} \, k!
  (m-k)! (n-m-1)!$.

  Assuming that $m=n-1$, (\ref{eq:ostar-tot}) follows from the bound given in
  (\ref{eq:otot1}), since a dominating set in $\gG$ is formed as soon as node $n$
  is picked and the chance that this has not happen in the first $l$ steps is equal
  to $(1 - 1/n)^l$. Finally, suppose that the starting sign vector $x$ has all its
  coordinates equal to $-$ and let $A$ be the set of all $y \in \{-, +\}^{n-1}$ having
  at least $(n-1)/2$ coordinates equal to $+$. An elementary calculation shows
  that after $n$ steps in the chain, the expected number of $+$ coordinates is equal
  to
    $$ \frac{n-1}{2} \left( 1 - \left(1 - \frac{2}{n} \right)^n \right) \ \sim \
       \frac{n-1}{2} \left( 1 - \frac{1}{e^2} \right). $$
  It follows that $(K^*_x)^n (A) \to 0$ as $n \to \infty$, while clearly $\bar{\pi}
  (A) \ge 1/2$. Since the total variation distance $\| (K^*_x)^n - \bar{\pi}
  \|_\TV$ is bounded below by $|(K^*_x)^n (A) - \bar{\pi} (A) |$, we conclude that
  given any $0 < \theta < 1/2$ we have $\| (K^*_x)^n - \bar{\pi} \|_\TV \ge \theta$
  for $n$ large enough. This completes the proof of part (iii). A similar argument
  works for all $2 \le m \le n-1$.
  \end{proof}
  }
  \end{example}

  \begin{example} [Descent Set]  \label{ex:aopath}
    {\rm Let $\gG$ be the path with edges $\{i, i+1\}$ for $1 \le i \le n-1$ and
    choose weights $w_1 = \cdots = w_n = 1/n$. Once again, every orientation of
    $\gG$ is acyclic and hence the set $\AO (\gG)$ can be identified with the
    set of sign vectors $\{-, +\}^{n-1}$. We leave it to the reader to give a
    description of the evolution of this chain on the set $\{-, +\}^{n-1}$
    similar to that of Example \ref{ex:aostar}.

    We find it more convenient to identify $\AO (\gG)$ with the set of subsets
    of $[n-1]$, where an orientation $\oo$ of $\gG$ is identified with the set
    of indices $i \in [n-1]$ for which the edge $\{i, i+1\}$ is directed as $i
    \to i+1$ in $\oo$. We denote by $\bB$ the graphical arrangement associated to
    $\gG$, as usual, and recall that there is a directed edge $i \to i+1$ in $\oo$
    if and only if $x_i < x_{i+1}$ holds in the chamber $D$ of $\bB$ corresponding
    to $\oo$. In turn, this happens if and only if $i+1$ precedes $i$ in any of
    the permutations $\tau$ which correspond to chambers of $\aA$ contained in $D$
    or, equivalently, if and only if $i$ belongs to the descent set
      $$ \Des (\tau^{-1}) \ = \ \{i \in [n-1]: \tau^{-1}(i) > \tau^{-1}(i+1)\} $$
    of the inverse permutation $\tau^{-1}$. Therefore, our chain on the set of
    subsets of $[n-1]$ is the process which records the descent set $\Des (\tau^{-1})$,
    when $\tau \in \sss_n$ evolves as in the Markov chain of Example
    \ref{ex:tsetlin} with uniform weights. We recall that a \emph{composition} of
    $n$ is an ordered sequence of positive integers (called parts) which sum to $n$.

  \begin{proposition} \label{prop:aopath}
  Consider the Markov chain on $\AO (\gG)$, which is induced from the Tsetlin library
  with uniform weights, as a chain on the set of subsets of $[n-1]$ and let $K^*$ be
  its transition matrix.
     \begin{itemize}
       \item[(i)] The matrix $K^*$ is diagonalizable with eigenvalues $j/n$
         for $j \in \{0, 1,\dots,n-2\} \cup \{n\}$, where the multiplicity of $j/n$ is
         equal to the number of compositions of $n$ having exactly $j$ parts equal to 1.
       \item[(ii)] The stationary distribution of $K^*$ is given by
         \begin{equation} \label{eq:piDes}
           \bar{\pi} (S) \ = \ \frac{1}{n!} \ \# \, \{\tau \in \sss_n: \Des (\tau)
           = S \}
         \end{equation}
       for every $S \subseteq [n-1]$.
       \item[(iii)] We have
         $$ \| (K^*_S)^l - \bar{\pi} \|_\TV \ \le \ (n-1) \left(1 - \frac{2}{n}
            \right)^l, $$
       where $(K^*_S)^l$ is the distribution of the chain started from $S$ after $l$
       steps.
     \end{itemize}
  \end{proposition}
  \begin{proof}
  Part (i) follows once again from Corollary \ref{cor:o} (i), or from the proof of
  Proposition \ref{prop:tsetlin} (i). For part (ii) it suffices to note that given
  $S \subseteq [n-1]$ with corresponding orientation $\oo \in \AO(\gG)$, the set of
  linear extensions of the partial order on $[n]$ defined by $\oo$ is in bijection
  with the set of elements of $\sss_n$ with descent set equal to $S$, as already
  discussed before the statement of the proposition. Then (\ref{eq:piDes}) follows
  from (\ref{eq:aoupi}) and Proposition \ref{prop:tsetlin} (ii). Part (iii) is
  a consequence of (\ref{eq:otot3}).
  \end{proof}

  Part (iii) of this proposition implies that for $c>0$, the distance $\| (K^*_S)^l
  - \bar{\pi} \|_\TV$ is bounded above by $e^{-c}$ if $l \ge \frac{n}{2} (\log n +
  c)$. This can be shown to be sharp, in the sense of Proposition \ref{prop:aostar}
  (iii), by an argument similar to the one in the proof of this proposition.
  \qed}
  \end{example}

  \begin{example} [Cyclic Descent Set]  \label{ex:aocycle}
    {\rm For notational convenience, in this example we replace the node set $[n]$ by
    the abelian group $\ZZ_n$ of integers modulo $n$. We let $\gG$ be the cycle with
    edges $\{i, i+1\}$ for $1 \le i \le n$ and choose weights $w_1 = \cdots = w_n
    = 1/n$. Since there are exactly two orientations of $\gG$ which have a directed
    cycle, the number of acyclic orientations of $\gG$ is equal to $2^n - 2$. We
    may identify $\AO (\gG)$ with the set of proper subsets of $[n]$, where an acyclic
    orientation $\oo$ of $\gG$ corresponds to the set of indices $i \in [n]$ for which
    the edge $\{i, i+1\}$ is directed as $i \to i+1$ in $\oo$. Arguing as in Example
    \ref{ex:aopath}, we see that this chain on the set of proper subsets of $[n]$ is
    the process which records the cyclic descent set
     $$ \cDes (\tau^{-1}) \ = \ \{i \in [n]: \tau^{-1}(i) > \tau^{-1}(i+1)\} $$
    when $\tau \in \sss_n$ evolves as in the Markov chain of Example \ref{ex:tsetlin}
    with uniform weights. Cyclic descents of permutations were introduced by Cellini
    \cite{Ce} and further studied by Fulman; see \cite{Fu} and references therein.

    The bond lattice $\lL_\gG$ is isomorphic to the set of subsets of $[n]$, other
    than those of cardinality $n-1$, partially ordered by inclusion. Since this
    lattice has a well known M\"obius function, one can deduce easily from
    (\ref{eq:lamdasigma}) the following description of the eigenvalues of the
    transition matrix $K^*$ of this chain. They are the numbers $j/n$ for $j
    \in \{0, 1,\dots,n\}$ and for $j \ge 1$, the multiplicity of $j/n$ is equal to the
    number of set partitions of $\ZZ_n$ into blocks of the form $\{a, a+1,\dots,b\}$
    having exactly $j$ singleton blocks. The multiplicity of zero is two less
    than the number of such partitions of $\ZZ_n$ having no singleton block. Arguing
    as in Example \ref{ex:aopath}, we find that the stationary distribution of $K^*$
    is given by
         \begin{equation} \label{eq:picDes}
           \bar{\pi} (S) \ = \ \frac{1}{n!} \ \# \, \{\tau \in \sss_n: \cDes (\tau)
           = S \}
         \end{equation}
    for proper subsets $S \subseteq [n]$ and that
      \begin{equation} \label{eq:tot5}
        \| (K^*_S)^l - \bar{\pi} \|_\TV \ \le \ n \left(1 - \frac{2}{n} \right)^l,
      \end{equation}
    where $(K^*_S)^l$ is the distribution of the chain started from $S$ after $l$
    steps. As in Example \ref{ex:aopath}, it follows that  $\| (K^*_S)^l - \bar{\pi}
    \|_\TV$ is bounded above by $e^{-c}$ if $l \ge \frac{n}{2} (\log n + c)$.
  \qed}
  \end{example}

  \medskip
  \noindent
  \textbf{B. Inverse $a$-shuffling}.
  Let $a \ge 2$ be an integer and let $w$ be the probability measure on $\fF_\aA$
  of Example \ref{ex:a}, so that the hyperplane walk associated to $w$ is the
  Markov chain of inverse $a$-shuffles on $\sss_n$. Using similar reasoning to the
  one in the case of the Tsetlin library, one can describe the induced chain of
  Proposition \ref{prop:ao} on the set $\AO (\gG)$ as follows: The chain proceeds
  from a given acyclic orientation of $\gG$ by selecting uniformly at random a
  weak ordered partition $B = (B_1, B_2,\dots,B_a)$ of $[n]$ with $a$ blocks.
  Then the orientation of any edge of $\gG$ whose endpoints belong to the same
  block of $B$ is left unchanged and any other edge $\{u, v\}$ of $\gG$ is
  reoriented as $u \to v$, if $i < j$ holds for the unique indices $i$ and $j$ with
  $v \in B_i$ and $u \in B_j$, to reach a new acyclic orientation of $\gG$. We will
  refer to the induced chain as the chain of inverse $a$-shuffles on $\AO (\gG)$.
  Its transition matrix $K^*$ satisfies
      $$ K^* (\oo, \oo') \ = \ \frac{\nu^* (\oo, \oo') }{a^n}$$
  for $\oo, \oo' \in \AO (\gG)$, where $\nu^* (\oo, \oo')$ is the number of weak
  ordered partitions of $[n]$ with $a$ blocks, the action of which on $\oo$, just
  described, results in $\oo'$.

  We denote by $\chi_\gG$ the chromatic polynomial
  \cite[Section 2.3]{Sta3} of $\gG$. Thus for every positive integer $q$,
  $\chi_\gG (q)$ is equal to the number of colorings $\kappa: [n] \to [q]$ of the
  nodes of $\gG$ with $q$ colors satisfying $\kappa (u) \ne \kappa (v)$ for every
  edge $\{u, v\}$ of $\gG$. Part (i) of the following corollary provides an
  interpretation to the coefficients of $\chi_\gG$ which strengthens a theorem of
  Stanley \cite{Sta1} \cite[Corollary 2.3]{Sta3}, stating that the sum of the
  unsigned coefficients of $\chi_\gG$ is equal to the number of acyclic orientations
  of $\gG$. There are other interpretations to these coefficients; see, for
  instance, \cite[Theorem 4.12]{Sta3} and \cite{Ha, Ste}.

  \begin{proposition} \label{prop:a}
  Let $\gG$ be a simple graph on the node set $[n]$ and let $K^*$ be the transition
  matrix of the Markov chain of inverse $a$-shuffles on $\AO (\gG)$.
     \begin{itemize}
       \item[(i)] The matrix $K^*$ is diagonalizable with characteristic
       polynomial given by
       $$\det (xI - K^*) \ = \ \prod_{i=0}^{n-1} \ (x - \frac{1}{a^i})^{p_i},$$
       where
       $$\chi_\gG (q) \ = \ \sum_{i=0}^{n-1} \ (-1)^i \, p_i q^{n-i}$$
       is the chromatic polynomial of $\gG$.
       \item[(ii)] The  stationary distribution $\bar{\pi}$ of $K^*$ is given
       by (\ref{eq:aoupi}).
       \item[(iii)] We have
         \begin{equation} \label{eq:tota}
         \| (K^*_\oo)^l - \bar{\pi} \|_\TV \ \le \ m \left( \frac{1}{a} \right)^l,
         \end{equation}
       where $m$ is the number of edges of $\gG$ and $(K^*_\oo)^l$ is the
       distribution of the chain started from the acyclic orientation $\oo$,
       after $l$ steps.
     \end{itemize}
  \end{proposition}
  \begin{proof}
    Let $\bB$ denote the graphical arrangement corresponding to $\gG$, as usual,
    and $\mu_\bB$ denote the M\"obius function of the intersection poset $\lL_\bB$.
    It follows from Theorem \ref{cor:main} (i) and (\ref{eq:1/a}) that the
    distinct eigenvalues of $K^*$ are $1, 1/a, 1/a^2,\dots,1/a^{n-1}$ and that
    the multiplicity $p_i$ of the eigenvalue $1/a^i$ satisfies
      $$ p_i \ = \ \sum_{W \in \lL_\bB: \, \codim(W,V) = i} \ (-1)^i \, \mu_\bB
         (V, W). $$
    Equivalently, $(-1)^i p_i$ is equal to the coefficient of $q^{n-i}$ in the
    characteristic polynomial \cite[Section 2.3]{OT} \cite[Section 1.3]{Sta3} of
    $\bB$, which is known to equal $\chi_\gG (q)$ \cite[Theorem 2.7]{Sta3}. This
    proves part (i). Part (ii) follows from Proposition \ref{prop:ao}, since the
    chain of inverse $a$-shuffles converges to the uniform distribution on $\sss_n$.
    Part (iii) follows from (\ref{eq:tot4}), since $\lambda_H = 1/a$ for every $H
    \in \bB$.
  \end{proof}

  For the applications discussed in the following example, we think of $\sss_n$
  as the set of linear orderings of a deck of $n$ cards, labeled by the elements
  of $[n]$. Since inverse $a$-shuffling, followed by passing to the inverse of
  the current permutation, gives the same distribution as ordinary $a$-shuffling,
  there is a straightforward translation of our results into the language of
  $a$-shuffles.

  \begin{example} \label{ex:aoa}
    {\rm (i) Suppose that $\gG$ is the star of Example \ref{ex:aostar}, say
    with $n-1$ edges. The chain of inverse $a$-shuffles on $\AO (\gG)$ is the
    process which records the set $S \subseteq [n-1]$ of the labels of cards which
    precede card $n$ in the current linear ordering, in the chain of inverse
    $a$-shuffles on $\sss_n$ (this is stronger than just recording the current
    position of card $n$; see \cite{ADS} for a summary of results on that Markov
    chain).

    Since $\gG$ is a tree, its chromatic polynomial is given by $\chi_\gG (q) =
    q (q-1)^{n-1}$ and hence, by Proposition \ref{prop:a} (i), the matrix $K^*$ has
    eigenvalues $1/a^i$ with multiplicity ${n-1 \choose i}$, for $0 \le i \le n-1$.
    The stationary distribution $\bar{\pi} (S)$ is given by the right-hand side of
    (\ref{eq:pix}), where $m=n-1$ and $k$ is the number of elements of $S$.
    Moreover, (\ref{eq:tota}) gives
       \begin{equation} \label{eq:tot-a-star}
        \| (K^*)^l - \bar{\pi} \|_\TV \ \le \ (n-1) \left( \frac{1}{a} \right)^l
       \end{equation}
    and hence the total variation distance on the left is bounded above by $a^{-c}$
    if $l \ge \log_a n + c$, for $c>0$. This shows a speedup over the $(3/2) \log_a
    n + c$, required for the parent chain of inverse $a$-shuffles on $\sss_n$ to
    reach stationarity, and is essentially sharp by the results of \cite[Section
    2]{ADS}.

    (ii) Suppose that $\gG$ is the path of Example \ref{ex:aopath}. The induced
    chain records the descent set $\Des (\tau^{-1})$ of the inverse of the
    current permutation $\tau$ in the chain of inverse $a$-shuffles on $\sss_n$. The
    stationary distribution is given by (\ref{eq:piDes}). Since the path $\gG$ is
    also a tree, the description of the eigenvalues for the star example and
    (\ref{eq:tot-a-star}) continue to hold. The result on the rate of convergence
    in this case was obtained earlier in \cite[Section 3]{DiFu}, where it is also
    shown that $(1/2) \log_a n + c$ steps are necessary and sufficient for $\Des
    (\tau^{-1})$ to reach stationarity if $n$ is large.

    (iii) Suppose that $\gG$ is the cycle of Example \ref{ex:aocycle}. The induced
    chain now records the cyclic descent set $\cDes (\tau^{-1})$ of the inverse of
    the current permutation $\tau$ in the chain of inverse $a$-shuffles on $\sss_n$.
    It follows, in particular, that this process is a Markov chain on the set of
    proper subsets of $[n]$. As in the previous case, we find that the stationary
    distribution is given by (\ref{eq:picDes}) and that
      $$ \| (K^*)^l - \bar{\pi} \|_\TV \ \le \ n \left( \frac{1}{a} \right)^l, $$
    so that $\| (K^*)^l - \bar{\pi} \|_\TV$ is bounded above by $a^{-c}$ if $l \ge
    \log_a n + c$, for $c>0$. We leave further details to the interested reader.
  \qed}
  \end{example}

  \begin{remark} \label{rem:o-inv}
    {\rm  The difference between ordinary and inverse $a$-shuffles is easy to
    appreciate by considering the graph with a single edge $\{1, 2\}$. Then the
    induced process records the relative order of cards labeled 1 and 2. After
    fewer than $\log_a n$ ordinary $a$-shuffles, there is a good chance (close to
    1) that they are still in their original order. However, their relative order
    is close to random after a growing number of inverse $a$-shuffles. \qed}
  \end{remark}

  \section{Semigroup walks}
  \label{sec:semi}

  The theory of hyperplane walks was generalized to random walks on semigroups by
  Brown \cite{Br1, Br2}. This section shows how our main results can be extended in
  this direction. Some familiarity with the ideas of \cite{Br1, Br2} will be
  assumed. The algebraic aspects of Brown's theory of semigroup walks have been
  further studied in \cite{AM, Stb1, Stb2}, with probabilistic developments in
  \cite{DS}. These references contain examples to which the following theory may be
  applied.

  The \emph{face semigroup} of a hyperplane arrangement $\aA$ is defined as the
  set $\fF_\aA$ of faces of $\aA$, endowed with the product operation discussed in
  Section \ref{sec:back}. The set $\cC_\aA$ of chambers is a left ideal of $\fF_\aA$,
  meaning that it is a nonempty subset which is closed under left multiplication by
  elements of $\fF_\aA$ (of course, $\cC_\aA$ is a right ideal as well). Given a
  finite semigroup $\sS$, a left ideal $\cC$ of $\sS$ and a probability measure $w$
  on $\sS$, one can define a Markov chain on the state space $\cC$ with transition
  matrix $K$ given by
    \begin{equation} \label{eq:semK}
     K(c, c') \ = \ \sum_{x \in \sS: \, xc=c'} \ w(x)
    \end{equation}
  for $c, c' \in \cC$. We refer to this chain as the \emph{semigroup walk} on $\cC$
  associated to $w$; it coincides with the hyperplane walk on $\aA$
  associated to $w$, if $\sS = \fF_\aA$ and $\cC = \cC_\aA$. The semigroup $\sS$ is
  said to be a \emph{band} if $x^2 = x$ for every $x \in \sS$. To such a semigroup,
  one can associate a join semilattice $L$ and a surjective map $\supp: \sS \to L$,
  satisfying
    \begin{equation} \label{eq:supp}
      \supp (x) \le_L \supp (y) \ \Leftrightarrow \ y=yxy
    \end{equation}
  for $x, y \in \sS$; see \cite[Section A.2]{Br2} for further details. The support map
  has the additional property that
    \begin{equation} \label{eq:join}
      \supp (xy) \ = \ \supp (x) \vee \supp (y)
    \end{equation}
  for $x, y \in \sS$, where $u \vee v$ denotes the least upper bound (join) of $u$
  and $v$ in $L$. In the special case of a face semigroup $\fF_\aA$, the support
  of a face $F \in \fF_\aA$ is the linear span of $F$ and the semilattice $L$ is
  the dual of the intersection poset $\lL_\aA$, defined in Section \ref{sec:back}.
  A band $\sS$ is called \emph{left-regular} if $xyx = xy$ for all $x, y \in \sS$.

  Assume that $\sS$ is a finite band. Then $L$ is a finite join semilattice and
  hence it has a maximum element, denoted $\hat{1}$. It follows from (\ref{eq:join})
  that the set $\cC_\sS = \{c \in \sS: \supp (c) = \hat{1}\}$ is a left ideal of
  $\sS$. The elements of $\cC_\sS$ are called \emph{chambers}. Given $x \in \sS$, the
  subsemigroup $x \sS$ is a finite band whose number of chambers depends only on
  the support $u = \supp (x)$ of $x$ in $L$; see \cite[Section B.3]{Br2}. We denote
  this number by $n(u)$. The following theorem was proved for left-regular bands in
  \cite{Br1} and generalized to all bands in \cite{Br2}.

  \begin{theorem} \label{thm:Brown}
    Let $\sS$ be a finite band with corresponding semilattice $L$ and set of
    chambers $\cC_\sS$. Let $w$ be a probability measure on $\sS$ and let $K$ be the
    transition matrix of the semigroup walk on $\cC_\sS$ associated to $w$.
     \begin{itemize}
        \item[(i)] The characteristic polynomial of $K$ is given by
        \begin{equation} \label{eq:sem-char}
          \det (xI - K) \ = \ \prod_{u \in L} \ (x - \lambda_u)^{m_u},
        \end{equation}
        where
          \begin{equation} \label{eq:sem-lambdaW}
            \lambda_u \ = \ \sum_{x \in \sS: \ \supp (x) \le_L u} \ w(x)
          \end{equation}
        is an eigenvalue,
        \begin{equation} \label{eq:sem-mu}
          m_u \ = \ \sum_{u \le_L v} \ \mu_L (u, v) \, n(v),
        \end{equation}
        $\mu_L$ is the M\"obius function of $L$ and $n(v)$ is the number of chambers
        of $x \sS$ for any $x \in \sS$ with $\supp (x) = v$.

        \item[(ii)] The matrix $K$ is diagonalizable.
        \item[(iii)] If the set $\{x \in \sS: w(x) > 0\}$ generates $\sS$, then
        $K$ has a unique stationary distribution $\pi$ and
          \begin{equation} \label{eq:sem-tot}
          \| K_c^l - \pi \|_\TV \ \le \ P \{x_1 x_2 \cdots x_l \not\in \cC_\sS \}
          \ \le \ \sum_{u} \ \lambda_u^l,
          \end{equation}
        where $K_c^l$ is the distribution of the chain started from $c \in \cC_\sS$
        after $l$ steps, $(x_1, x_2,\dots)$ consists of independent and identically
        distributed picks from the measure $w$ on $\sS$ and $u$ runs through the set
        of elements of $L$ covered by $\hat{1}$.
     \end{itemize}
  \end{theorem}

  The results of Section \ref{sec:main} can be extended to this setting as follows.
  Let $\phi: \sS_\aA \to \sS_\bB$ be an epimorphism of finite semigroups, meaning that
  $\phi$ is a surjective map which satisfies $\phi(xy) = \phi(x) \phi(y)$ for all $x, y
  \in \sS_\aA$. Given a left ideal $\cC$ of $\sS_\aA$ and a probability measure $w$ on
  $\sS_\aA$, the semigroup walk (\ref{eq:semK}) on $\cC$ associated to $w$ induces a
  stochastic process on the state space $\phi(\cC)$, in the sense of
  Section \ref{sec:functions}. Since $\phi$ is surjective, the image $\phi(\cC)$ is a left
  ideal of $\sS_\bB$. This setup generalizes that of the map $f: \fF_\aA \to \fF_\bB$ of
  face semigroups of Section \ref{sec:main}, where $\bB$ is a subarrangement of a
  hyperplane arrangement $\aA$ and $f(F)$ is the unique face of $\bB$ which contains $F$,
  for $F \in \fF_\aA$. The following proposition generalizes Corollary \ref{cor:sub}.
    \begin{proposition} \label{prop:sem}
      Let $\phi: \sS_\aA \to \sS_\bB$ be an epimorphism of semigroups, $\cC \subseteq
      \sS_\aA$ be a left ideal and $w$ be a probability measure on $\sS_\aA$. For every
      starting distribution on $\cC$, the stochastic process on $\phi(\cC)$ which is
      induced from the semigroup walk on $\cC$ associated to $w$ by the map $\phi$ is
      Markov. Moreover, this induced chain is itself a semigroup walk on $\phi(\cC)$,
      with associated probability measure $w^*$ on $\sS_\bB$ defined by
        \begin{equation} \label{eq:sem-w*}
          w^* (z) \ = \ \sum_{x \in \sS_\aA: \, \phi(x) = z} \ w(x).
        \end{equation}
    \end{proposition}
    \begin{proof}
    This follows by computations similar to those in the proofs of Proposition
    \ref{prop:sub} and Corollary \ref{cor:sub}.
    \end{proof}

  Suppose now that $\phi: \sS_\aA \to \sS_\bB$ is an epimorphism of finite bands. The
  definition of the support semilattice in \cite[Section A.2]{Br2} and (\ref{eq:supp})
  imply that $\phi$ induces an order preserving, surjective map $\phi_\ast: L_\aA \to
  L_\bB$ of the associated semilattices which makes the diagram
  $$ \begin{CD}
       \sS_\aA   @>{\phi}>>  \sS_\bB  \\
       @V{\supp_\aA}VV                              @VV{\supp_\bB}V     \\
       L_\aA     @>{\phi_\ast}>>              L_\bB    &  \\
  \end{CD} $$
  commute, where $\supp_\aA$ and $\supp_\bB$ are the support maps of $\sS_\aA$ and
  $\sS_\bB$, respectively. We denote by $\cC_\aA$ and $\cC_\bB$ the set of chambers of
  $\sS_\aA$ and $\sS_\bB$, respectively. Since $\cC_\aA$ is a left ideal of $\sS_\aA$,
  the image $\phi(\cC_\aA)$ is a left ideal of $\sS_\bB$.
    \begin{lemma} \label{lem:phi}
      Let $\sS_\aA$ and $\sS_\bB$ be finite bands with sets of chambers $\cC_\aA$ and
      $\cC_\bB$, respectively, and let $\phi: \sS_\aA \to \sS_\bB$ be an epimorphism
      of semigroups.
        \begin{itemize}
        \item[(i)] We have $\phi(\cC_\aA) \subseteq \cC_\bB$.

        \item[(ii)] If $\sS_\aA$ is left-regular, then $\phi(\cC_\aA) = \cC_\bB$.
        \end{itemize}
    \end{lemma}
    \begin{proof}
    To prove (i), suppose that $c \in \cC_\aA$. Then we have $\supp_\aA (x) \le \supp_\aA
    (c)$ in $L_\aA$ for every $x \in \sS_\aA$. By (\ref{eq:supp}), this means that $c = cxc$
    holds in $\sS_\aA$ for every $x \in \sS_\aA$. Since $\phi$ is an epimorphism of
    semigroups, it follows that $\phi(c) = \phi(c)z\phi(c)$ for every $z \in \sS_\bB$. By
    reversing the first part of the argument, we conclude that $\phi(c) \in \cC_\bB$.

    Assume now that $\sS_\aA$ is left-regular. It was shown in \cite[Sections 2.2 and
    B.3]{Br1} that the relation $\preceq_\aA$, defined by letting $x \preceq_\aA y
    \Leftrightarrow xy = y$ for $x, y \in \sS_\aA$, is a partial order on $\sS_\aA$ and
    that the chambers of $\sS_\aA$ are precisely the maximal elements of $\preceq_\aA$.
    Similar remarks hold for the band $\sS_\bB$, which is also left-regular as a
    homomorphic image of $\sS_\aA$. To prove
    (ii), suppose that $d \in \cC_\bB$ and let $x \in \sS_\aA$ be such that $\phi (x) = d$.
    Then there exists $c \in \cC_\aA$ such that $x \preceq_\aA c$. Clearly, the map $\phi:
    \sS_\aA \to \sS_\bB$ is order preserving and hence $d \preceq_\bB \phi(c)$. Since $d
    \in \cC_\bB$ is maximal in $\preceq_\bB$, we must have $d = \phi(c)$. This shows that
    $d \in \phi(\cC_\aA)$ and hence that $\cC_\bB \subseteq \phi(\cC_\aA)$. In view of part
    (i), it follows that $\phi(\cC_\aA) = \cC_\bB$.
    \end{proof}

  Lemma \ref{lem:phi} implies that if $\phi: \sS_\aA \to \sS_\bB$ is an epimorphism of
  finite left-regular bands and $\cC = \cC_\aA$, then the induced Markov chain of
  Proposition \ref{prop:sem} is a semigroup walk on the state space $\cC_\bB$ of chambers
  of $\sS_\bB$. Thus all conclusions of Theorem \ref{thm:Brown} apply to the induced chain.
  We leave it to the reader to formulate the exact analogue of Theorem \ref{cor:main} in
  this situation and end with a remark on the rest of the material of Section
  \ref{sec:main}.

  \begin{remark} \label{rem:semproofs}
    {\rm The proofs of parts (i) and (iv) of Theorem \ref{thm:BHR-BD}, given in
    Section \ref{sec:main}, extend easily in the setup of Theorem \ref{thm:Brown}.
    For part (i), for instance, one should replace (\ref{eq:FC}) by the equality
      \begin{equation} \label{eq:sem-FC}
        \# \{c \in \cC_\sS: xc = c\} \ = \sum_{u \in L: \ \supp(x) \le_L u} \ m_u
      \end{equation}
    for $x \in \sS$, where the $m_u$ are defined by (\ref{eq:sem-mu}). A computation
    similar to that given in the proof of Theorem \ref{thm:BHR-BD} (i) in Section
    \ref{sec:main}, equation (\ref{eq:join}) and a slight variant of Lemma
    \ref{lem:eigen} then imply that the $m_u$ are necessarily nonnegative integers
    and that (\ref{eq:sem-char}) holds. To check the validity of (\ref{eq:sem-FC}),
    we observe that the set $\{c \in \cC_\sS: xc = c\}$ is equal to the set of chambers
    of $x \sS$ (see, for instance, \cite[Example A.13]{Br2}). Thus the left-hand side
    of (\ref{eq:sem-FC}) equals $n(v)$, where $v = \supp (x) \in L$, and hence
    (\ref{eq:sem-FC}) is equivalent to
      $$ n(v) \ = \ \sum_{v \le_L u} \ m_u $$
    for $v \in L$. This is in turn equivalent to (\ref{eq:sem-mu}) by M\"obius
    inversion on $L$.
    We leave the details of the coupling proof of (\ref{eq:sem-tot}) to the interested
    reader.
    \qed}
  \end{remark}

  \section*{Acknowledgements}
  Athanasiadis was partially supported by the 70/4/8755 ELKE Research Fund of the
  University of Athens. Diaconis was partially supported by NSF grant DMS-0505673.

  \end{document}